%% file: main.tex
\documentclass{amsart}
\usepackage{amsmath}
\usepackage{amsthm}
\usepackage{mathtools}
\usepackage{amssymb}
\usepackage{mathrsfs}
\usepackage[misc]{ifsym}
  \usepackage{paralist}
  \usepackage{graphics} 
  \usepackage{epsfig} 
\usepackage{graphicx}  \usepackage{epstopdf}
 \usepackage[colorlinks=true]{hyperref}
\usepackage{comment}
\usepackage{stackrel}

\hypersetup{urlcolor=blue, citecolor=red}

\newtheorem{theorem}{Theorem}[section]
\newtheorem{corollary}{Corollary}
\newtheorem{maintheorem}{Theorem}

\newtheorem{lemma}[theorem]{Lemma}
\newtheorem{proposition}{Proposition}

\theoremstyle{definition}
\newtheorem{definition}[theorem]{Definition}

\newcommand\restr[2]{{ 
  \left.\kern-\nulldelimiterspace 
  #1 
  \vphantom{\big|} 
  \right|_{#2} 
  }}

\title[Rigidity for Anosov Endomorphisms on T2] 
      {Anosov Endomorphisms on the 2-torus:\\ Regularity of foliations and rigidity}

\subjclass{Primary: 37C15; Secondary: 37D20.}
 \keywords{Smooth dynamics, Hyperbolic dynamics, Non-invertible dynamics.}

\author{Marisa Cantarino}
\address{School of Mathematical Sciences, Monash University,\\ Clayton, VIC 3800, Australia.}
\email{\Letter \; marisa.cantarino@monash.edu }
\thanks{M. C. was partially financed by the Coordenação de Aperfeiçoamento de Pessoal de Nível Superior - Brasil (CAPES) - grant 88882.333632/2019-01 and the Fundação Carlos Chagas Filho de Amparo à Pesquisa of the State of Rio de Janeiro (FAPERJ) E-26/202.014/2022.}

\author{Régis Varão}
\address{Instituto de Matemática, Estatística e Computação Científica,\\
Universidade Estadual de Campinas - UNICAMP\\
Rua S\'ergio Buarque de Holanda, 651, Cidade Universitária\\
Campinas - SP, Brasil.}
\email{varao@unicamp.br}

\thanks{R. V. was partially financed by CNPq and Fapesp grants 18/13481-0 and 17/06463-3.}

\thanks{This is an author-created, un-copyedited version of an article accepted for publication in Nonlinearity.
The publisher is not responsible for any errors or omissions in this version of the manuscript or any version derived from it. The Version of Record is available online at \url{https://doi.org/10.1088/1361-6544/acf267}.}

\begin{document}

\begin{abstract}
 We provide sufficient conditions for smooth conjugacy between two Anosov endomorphisms on the 2-torus. From that, we also explore how the regularity of the stable and unstable foliations implies smooth conjugacy inside a class of endomorphisms including, for instance, the ones with constant Jacobian. As a consequence, we have in this class a characterization of smooth conjugacy between special Anosov endomorphisms (defined as those having only one unstable direction for each point) and their linearizations.
\end{abstract}

\maketitle

\section{Introduction} \label{sec:intro}

In this work we study rigidity results for Anosov endomorphisms: hyperbolic maps that are not necessarily invertible. The term \textit{rigidity} is associated with the idea that the value of an invariant or a specific property of the system determines its dynamics. In our case, we want to determine the smooth conjugacy class of the system, and the invariant will be given by its Lyapunov exponents.

Let us recall briefly the context for the invertible case. Given a closed manifold $M$, for a $C^1$ Anosov diffeomorphism $f: M \to M$, any nearby $C^1$ diffeomorphism $g$ is topologically conjugate to $f$, that is, there exists $h: M \to M$ such that $g \circ h = h \circ f$. This conjugacy $h$ is Hölder continuous, but if $h$ is $C^1$ and $x$ a periodic point such that $f^p(x) = x$, then
    \begin{align*}
        D(h \circ f^p)_x &= D(g^p \circ h)_x\\
        \implies Df^p_x &= (Dh_x)^{-1} \circ Dg_{h(x)}^p \circ Dh_x.
    \end{align*}
Therefore, if the conjugacy between $f$ and $g$ is smooth, then the matrices $Df^p_x$ $Dg_{h(x)}^p$ are conjugate. The conjugacy of these matrices is a necessary condition for the smooth conjugacy, and it is natural to ask whether it is a sufficient condition. It turns out that this condition on the periodic points is the main property of the sufficient condition, as can be seen from the results obtained in the invertible setting \cite{de1992smooth, gogolev22c, gogolev2011local, gogolev2020local}.

By considering maps that do not have an inverse, we have a few interesting behaviors. In contrast to Anosov diffeomorphisms, Anosov endomorphisms are not structurally stable in general. Feliks Przytycki \cite{przytycki1976anosov} used the fact that we can perturb a linear Anosov endomorphism to create many different unstable directions to show that a perturbation of an Anosov endomorphism may not be conjugated to it, since a conjugacy should preserve stable and unstable manifolds. Additionally, he proved that Anosov endomorphisms are structurally stable if and only if they are invertible or expanding maps. 

We say that an Anosov endomorphism is \textit{special} if each point has only one unstable direction. Although an Anosov endomorphism may not be structurally stable, it is conjugated to its linearization if and only if it is special \cite{aoki1994topological, sumi1994linearization, moosavi2019classification}. 

For hyperbolic maps on surfaces, the conjugacy between $Df^p_x$ and $Dg_{h(x)}^p$ for corresponding periodic points $x$ and $h(x)$ is equivalent to $f$ and $g$ having the same Lyapunov exponents $\lambda^{u/s}_f(x) = \lambda^{u/s}_g(h(x))$ on these points. We prove that this condition is indeed sufficient to guarantee smooth conjugacy.

\begin{maintheorem}
    \label{teo:a}
    Let $f, g: \mathbb{T}^2 \to \mathbb{T}^2$ be $C^k$, $k \geq 2$, Anosov endomorphisms topologically conjugated by $h: \mathbb{T}^2 \to \mathbb{T}^2$ homotopic to $Id$. If the corresponding periodic points of $f$ and $g$ have the same Lyapunov exponents, then the conjugacy $h$ is $C^k$. In particular, if $f$ and $g$ are $C^\infty$, then $h$ is also $C^\infty$.
\end{maintheorem}

This kind of result was addressed for the invertible case on $\mathbb{T}^2$ by Rafael de la Llave, José Manuel Marco and Roberto Moriyón in a serie of works \cite{marco1987invariants1, de1987invariants2, marco1987invariants3, de1988invariants4, de1992smooth} and for $\mathbb{T}^3$ by Andrey Gogolev and Misha Guysinsky \cite{gogolev22c, gogolev2017bootstrap}, and it has counterexamples in higher dimensions \cite{de1992smooth}, in which more hypotheses are required \cite{gogolev2011local}, \cite{gogolev2020local}.

For the non-invertible setting, the above theorem extends some of the results obtained independently by Fernando Micena \cite[Theorem A, Theorem B]{micena2020rigidity}. On \cite[Theorem A]{micena2020rigidity}, $f$ and $g$ need not to be special, as in our case, but he requires the topological conjugacy to preserve SRB measures instead of the condition on Lyapunov exponents of periodic points. The core of the proof is \cite[Lemma 4.2]{micena2020rigidity}, where he obtains a candidate to smooth conjugacy as the limit of solutions of differential equations over local unstable manifolds. On \cite[Theorem B]{micena2020rigidity}, he requires, in addition to the condition on Lyapunov exponents of periodic points, that $f$ is strongly special, meaning that each point has only one unstable direction and that the stable manifold is dense. Then he uses a conformal metric on unstable manifolds to, again, solve some differential equations and find a candidate for $C^1$ conjugacy. Since it is $C^1$, it preserves SRB measures, and then he applies \cite[Theorem A]{micena2020rigidity}. Additionaly, after the submission of our work, in a recent preprint, Ruihao Gu and Yi Shi prove a similar result in \cite[Theorem 1.3]{gu-shi-2022}.

Similarly to the approaches for the invertible setting, we apply a regularity lemma by Jean-Lin Journé \cite{journe1988regularity} to ``spread'' the regularity of $h$ along transversal foliations to the same regularity in a whole neighborhood of a point. However, we have to do so locally, since $h$ is not invertible. Additionally, there are several subtleties along the steps of the proof, such as fixing inverses locally in a well-defined way.

Recently, Jinpeng An, Shaobo Gan, Ruihao Gu and Yi Shi \cite{an-gan-gu-shi-2022} proved that an non-invertible $C^{1+\alpha}$ Anosov endomorphism on $\mathbb{T}^2$ is special if and only if every periodic point for $f$ has the same stable Lyapunov exponent. In particular, $\lambda^s_f(x) \equiv \lambda^s_A$ for every $x \in \mathbb{T}^2$ periodic for $f$, where $A$ is the linearization of $f$. Then, if $f$ is special, we have the following corollary of Theorem \ref{teo:a}.

\begin{corollary}
    Let $f: \mathbb{T}^2 \to \mathbb{T}^2$ be a $C^k$, $k \geq 2$, special Anosov endomorphism and $A$ its linearization. If $\lambda^u_f(x) \equiv \lambda^u_A$ for every $x \in \mathbb{T}^2$ periodic for $f$, then the conjugacy $h$ between $f$ and $A$ is $C^k$. In particular, if $f$ is $C^\infty$, then $h$ is also $C^\infty$.
\end{corollary}

It remains to provide examples of non-special Anosov endomorphisms on $\mathbb{T}^2$ that are conjugated but with the conjugacy not being $C^1$, giving then conditions to apply Theorem \ref{teo:a} in its full generality. If $f$ and $g$ are conjugated, by lifting the conjugacy to the inverse limit space, $f$ and $g$ are inverse-limit conjugated. Then, a necessary condition is that $f$ and $g$ have the same linearization, since by Nobuo Aoki and Koichi Hiraide \cite[Theorem 6.8.1]{aoki1994topological} $f$ and its linearization are inverse-limit conjugated. Besides, the conjugacy must necessarily be a homeomorphism between $W^u_f(\tilde{x})$ and $W^u_g(\tilde{h}(\tilde{x}))$ for each of the unstable directions. Fernando Micena and Ali Tahzibi \cite{micena-tahzibi2016unstable} proved that, if $f$ is not special, there is a residual subset $\mathcal{R} \in \mathbb{T}^2$ such that every $x \in \mathcal{R}$ has infinitely many unstable directions. This suggests the complexity of this problem.

\textbf{Question:} Under which conditions do we have a topological conjugacy between two non-special Anosov endomorphisms on $\mathbb{T}^n$ with the same linearization?

An answer to this question for $\mathbb{T}^2$ is given in \cite[Theorem 1.1]{gu-shi-2022}, in which they prove that, if $f$ and $g$ are non-invertible and homotopic, then $f$ is topologically conjugate to $g$ if, and
only if, the corresponding periodic points of $f$ and $g$ have the same stable Lyapunov exponents. Then Theorem $A$ can be reformulated as follows.

\begin{corollary}
    Let $f, g: \mathbb{T}^2 \to \mathbb{T}^2$ be $C^k$, $k \geq 2$, homotopic Anosov endomorphisms. If the corresponding periodic points of $f$ and $g$ have the same Lyapunov exponents, then they are conjugated and the conjugacy $h$ is $C^k$. In particular, if $f$ and $g$ are $C^\infty$, then $h$ is also $C^\infty$.
\end{corollary}

To provide a context in which Theorem \ref{teo:a} can be applied, our approach is towards the regularity of the stable and unstable directions. This is inspired by a conjecture for Anosov diffeomorphisms on compact manifolds, which states that $C^2$ regularity of the stable and unstable foliations of an Anosov diffeomorphism would imply the same regularity for the conjugacy with the Anosov automorphism homotopic to it \cite{flaminio-katok, varao-DS}. A similar result for $C^{1+\beta}$ stable and unstable foliations cannot hold, since in surfaces the stable and unstable manifolds are $C^{1+\beta}$ but the conjugacy is generally not better than Hölder continuous.

Another inspiration is given by the fact that the absolute continuity of stable and unstable foliations is a central property used to prove the ergodicity of conservative Anosov diffeomorphisms \cite{anosov1969geodesic}. That is, a regularity condition on the foliation implies a very specific behavior of the map.

We then consider the following context: what properties may we require on the stable or unstable foliations of a special Anosov endomorphism to get regularity for the conjugacy map?

Instead of absolute continuity, we work with a uniform version of absolute continuity --- called the \emph{UBD (uniform bounded density)} property, as defined by F. Micena and A. Tahzibi \cite{micena2013regularity}, see Definition \ref{def:UBD}. In \cite{varao2018rigidity} it is shown that a smooth conservative partially hyperbolic diffeomorphism on $\mathbb T^3$ is smoothly conjugate to its linearization if and only if the center foliation has the UBD property. This can be seen as a sharp result, since it is given as an example on \cite{varao2016center} a conservative partially hyperbolic diffeomorphism on $\mathbb T^3$ such that the center foliation is $C^1$ but the conjugacy map is not $C^1$. Hence, a uniformity condition in the densities is the natural candidate for rigidity results. Since the center foliation does not exist in our context, we work with a regularity condition for the unstable foliation. In a work in preparation, Marisa Cantarino, Simeão Targino da Silva and Régis Varão \cite{holonomies} prove that the UBD property is equivalent do the holonomies having uniformly bounded Jacobians.

For the sake of simplicity, we state the results for endomorphisms with constant Jacobian as a particular case, and we treat the more general context on Section \ref{sec:proofb}, for a class of maps that satisfy a condition analogue to conservativeness.

\begin{maintheorem} \label{teo:b}
    Let $f: \mathbb{T}^2 \to \mathbb{T}^2$ be a $C^2$ Anosov endomorphism with constant Jacobian and $A$ its linearization. If the unstable foliation of $f$ is absolutely continuous with uniformly bounded densities, then $\lambda_f^{\sigma} \equiv \lambda_A^{\sigma}$ for $\sigma \in \{u, s\}$.
\end{maintheorem}

In the above result $f$ is not required to be special. Note that, in general, $f$ does not have a global unstable foliation, since each point may have more than one unstable leaf, so when we say ‘‘the unstable foliation of $f$ has the UBD property’’, we are actually looking at a foliation on the universal cover of $\mathbb{T}^2$ (see Definition \ref{def:UBD}). 

We introduce a more general condition than constant Jacobian for $f$, that we call \textit{quasi preservation of densities} (see Definition \ref{def:hyp-c}). It means that the conditional measures are ``controlled'' under iterations of $f$, which is automatic if the Jacobian is constant (see Lemma \ref{lem:claim1}). In Subsection \ref{sec:teoc}, we see that Theorem \ref{teo:b} is a special case of the following.

\begin{maintheorem} \label{teo:c}
    Let $f: \mathbb{T}^2 \to \mathbb{T}^2$ be a $C^\infty$ Anosov endomorphism with quasi preservation of densities along its invariant foliations, and $A$ its linearization. If the stable and unstable foliations of $f$ are absolutely continuous with uniformly bounded densities, then $\lambda_f^{\sigma} \equiv \lambda_A^{\sigma}$ for $\sigma \in \{u, s\}$.
\end{maintheorem}

By joining Theorems \ref{teo:a} and \ref{teo:b} for $f$ special and $g = A$, we have that our regularity condition on the unstable foliation implies equality of Lyapunov exponents, which implies that $h$ is as regular as $f$. Conversely, if $h$ is $C^\infty$, the unstable leaves of $F$ are taken to unstable lines of $A$ by a lift $H$ of $h$, which implies the UBD property of the unstable foliation of $F$. The same holds for the stable foliation, and Corollary \ref{teo:coro} follows.

\begin{corollary}
    \label{teo:coro}
    Let $f: \mathbb{T}^2 \to \mathbb{T}^2$ be a $C^\infty$ special Anosov endomorphism with constant Jacobian and let $A$ be its linearization. The unstable foliation of $f$ is absolutely continuous with uniformly bounded densities if and only if $f$ is $C^\infty$ conjugate to $A$.
\end{corollary}

Corollary \ref{teo:coro} is the natural formulation of a result similar to Theorem 1.1 from \cite{varao2018rigidity} for Anosov endomorphisms on $\mathbb{T}^2$. Indeed, working with preservation of volume for maps that are not invertible is more subtle: a conservative endomorphism may not have constant Jacobian. Instead, we ask for both foliations to have the UBD property and to have their induced volume densities \textit{quasi preserved under $f$} (see Definition \ref{def:hyp-c}), with the case with $f$ having constant Jacobian as a particular case. They are necessary conditions in Corollary \ref{teo:coro}: if an Anosov endomorphism $f$ is $C^1$ conjugated to its linearization, then the unstable and stable foliations have the UBD property and present quasi preservation of densities with respect to $f$, as we see in Section \ref{sec:proofb}.

\subsection{Structure of the paper and comments on the results}

In Section \ref{sec:preli} we introduce formally some of the aforementioned notions, as well as some properties necessary for the proofs. Regarding Theorem \ref{teo:a}, that we prove in Section \ref{sec:proofa}, we want to apply Journé's regularity lemma \cite{journe1988regularity} to take the regularity of the conjugacy on transverse foliations to the whole manifold. To avoid complications from the fact that $f$ and $g$ are not invertible, we will apply this lemma locally on the universal cover, and this requires some adaptations. We proceed (as done in \cite{gogolev22c} for the weak unstable direction) by constructing, for the unstable direction, a function $\rho$ and using it to obtain the regularity of $h$ along unstable leaves. For stable leaves, an analogous argument holds. Even knowing that there are SRB measures for Anosov endomorphisms \cite{qian2002srb}, their definition involves the inverse limit space, and to avoid it our definition of $\rho$ is adapted to a local foliated box projected from the universal cover. The function $\rho$ is, indeed, proportional to the density of the conditional of the SRB measure on the unstable leaves where it is well defined.

The proof of Theorem \ref{teo:a} relies on the low dimension of $\mathbb{T}^2$ to construct metrics on each unstable leaf that ``behave regularly'' under $f$. We use these metrics to prove that the conjugacy is uniformly Lipschitz restricted to the leaf in Lemma \ref{lem:teo2-1}. After that, we prove that the conjugacy restricted to each leaf is in fact $C^1$. Additionally, we promote this $C^1$ regularity to the regularity of $f$ and $g$ with the argument that the densities of the conditionals of the SRB measure on each unstable leaf are $C^{k-1}$, so we are using the fact that the unstable foliation is one dimensional to claim that our densities $\rho$ are $C^{k-1}$. Since the argument for the stable foliation is analogous, it cannot be easily transposed to higher dimensions. Besides, a similar theorem for $\mathbb{T}^n$, $n \geq 4$, would need additional hypothesis, as it includes the invertible case \cite{de1992smooth}.

For the proof of Theorem \ref{teo:b} in Section \ref{sec:proofb}, we need the stable and unstable distributions of a lift of an Anosov endomorphism to be $C^1$. For this, we see in Proposition \ref{prop:F-C1-hol} that a codimension one unstable distribution for a lift of an Anosov endomorphism to the universal cover is $C^1$, again making use of the fact that we are working on $\mathbb{T}^2$. The strategy to prove Theorem  \ref{teo:b} is to construct, along unstable leaf segments of the lift $F$, invariant measures uniformly equivalent to the length that grow at the same rate that the unstable leaves of $A$, then the unstable Lyapunov exponents of both $f$ and $A$ are the same. These measures are defined for leaves of a full volume set, and we use the regularity of stable holonomies --- provided by the fact that the stable distribution is $C^1$ --- to construct similar measures at each point and conclude that $\lambda_f^{u} \equiv \lambda_A^{u}$.

\section{Preliminary concepts}
\label{sec:preli}

Along this section, we work with concepts and results needed for the proofs on sections \ref{sec:proofa} and \ref{sec:proofb}.

\subsection{Inverse limit space}

For certain dynamical aspects, such as unstable directions, we need to analyze the past orbit of a point. Since every point has more than one preimage, there are several choices of past, and we can make each one of these choices a point on a new space, defined as follows.

\begin{definition}
Let $(X,d)$ be a compact metric space and $f: X \to X$ continuous. The \emph{inverse limit space} (or \emph{natural extension}) to the triple $X$, $d$ and $f$ is
\begin{itemize}
    \item $\Tilde{X} = \{\Tilde{x} = (x_k) \in X^\mathbb{Z}: x_{k+1} = f(x_k) \mbox{, } \forall k \in \mathbb{Z}\}$,
    \item $(\Tilde{f}(\Tilde{x}))_k = x_{k+1}$ $\forall k \in \mathbb{Z}$ e $\forall \Tilde{x} \in \Tilde{X}$, 
    \item $\Tilde{d}(\Tilde{x}, \Tilde{y}) = \sum\limits_k \dfrac{d(x_k, y_k)}{2^{|k|}}$.
\end{itemize}
\end{definition}

With this definition, $(\Tilde{X},\Tilde{d})$ is a compact metric space and the shift $\Tilde{f}$ is continuous and invertible. If $\pi: \tilde{M} \to M$ is the projection on the 0th coordinate, $\pi(\tilde{x}) = x_0$, then $\pi$ is continuous.

There have been several advances on the study of non-invertible systems by making use of the inverse limit space. To name a few, it is the natural environment to search for structural stability \cite{berger2013inverse, berger2017structural}, and it provides tools to explore the measure-theoretical properties of these systems \cite{qian2009smooth}. Although the inverse limit space is not a manifold in general, in the Axiom A (which includes the case of Anosov Endomorphisms) case it can be stratified by finitely many laminations whose leaves are unstable sets \cite[\S 4.2]{berger2013inverse}.

In the torus case, the inverse limit space has a very specific algebraic and topological structure, more precisely, the inverse limit space has the structure of a compact connected abelian group with finite topological dimension, which is called a \emph{solenoidal group} \cite[\S 7.2]{aoki1994topological}. It can also be seen as a fiber bundle $(\tilde{X}, X, \pi, C)$, where the fiber $C$ is a Cantor set \cite[Theorem 6.5.1]{aoki1994topological}.

\subsection{Anosov endomorphisms: Some properties}

Let $M$ be a closed $C^\infty$ Riemannian manifold. 

\begin{definition}[\cite{przytycki1976anosov}]
Let $f: M \to M$ be a local $C^1$ diffeomorphism. $f$ is an \emph{Anosov endomorphism} (or \emph{uniformly hyperbolic endomorphism}) if, for all $\tilde{x} \in \tilde{M}$, there is, for all $i \in \mathbb{Z}$, a splitting $T_{x_i}M = E^u(x_i) \oplus E^s(x_i)$ such that
\begin{itemize}
    \item $Df(x_i)E^{u}(x_i) = E^{u}(x_{i+1})$;
    \item $Df(x_i)E^{s}(x_i) = E^{s}(x_{i+1})$;
    \item there are constants $c > 0$ and $\lambda >1$ such that, for a Riemannian metric on $M$,
    
    $||Df^{n}(x_i)v|| \geq c^{-1} \lambda^{n} ||v||$, $\forall v \in E^u(x_i)$, $\forall i \in \mathbb{Z}$,

$||Df^{n}(x_i)v|| \leq c \lambda^{-n} ||v||$, $\forall v \in E^s(x_i)$, $\forall i \in \mathbb{Z}$.
\end{itemize}
$E^{s/u}(x_i)$ is called the \emph{stable/unstable direction} for $x_i$. When it is needed to make explicit the choice of past orbit, we denote by $E^u(\tilde{x})$ the unstable direction at the point $\pi(\tilde{x}) = x_0$ with respect to the orbit $\tilde{x}$.
\end{definition}

This definition includes Anosov diffeomorphisms (if $f$ is invertible) and expanding maps (if $E^u(x) = T_{x}M$ for each $x$). Throughout this work, however, we consider the case in which $E^s$ is not trivial, excluding the expanding case, unless it is mentioned otherwise.

A point can have more than one unstable direction under an Anosov endomorphism, even though the stable direction is always unique. Indeed, we find in \cite{przytycki1976anosov} an example in which a point has uncountable many unstable directions.

Note that there is not a global splitting of the tangent space; the splitting is along a given orbit $\tilde{x} = (x_i)$ on $\tilde{M}$. But $\tilde{x}$ induces a hyperbolic sequence, and one can apply the Hadamard--Perron theorem in the same way it is done for Anosov diffeomorphisms to prove that $f$ has \emph{local stable} and \emph{local unstable manifolds}, denoted by $W_{f,R}^s(\tilde{x})$ and $W_{f,R}^s(\tilde{x})$, tangent to the stable and unstable directions \cite[Theorem 2.1]{przytycki1976anosov}. Additionally, for $R > 0$ sufficiently small, these manifolds are characterized by
$$W^s_{f,R}(\tilde{x}) = \{y \in M: \forall k \geq 0 \mbox{, } d(f^k(y), f^k(x_0)) < R\}$$
and
$$W^u_{f,R}(\tilde{x}) = \{y \in M: \exists \tilde{y} \in \tilde{M} \mbox{ such that } \pi(\tilde{y}) = y \mbox{ e }\forall k \geq 0 \mbox{, } d(y_{-k}, x_{-k}) < R\}.$$
The \emph{global stable/unstable manifolds} are
$$W^s_{f}(\tilde{x}) = \{y \in M: d(f^k(y), f^k(x_0)) \xrightarrow[]{k \to \infty} 0\}$$
and
$$W^u_{f}(\tilde{x}) = \{y \in M: \exists \tilde{y} \in \tilde{M} \mbox{ such that } \pi(\tilde{y}) = y \mbox{ e } d(y_{-k}, x_{-k})  \xrightarrow[]{k \to \infty} 0\}.$$
Moreover, these manifolds are as regular as $f$. The stable manifolds do not depend on the choice of past orbit for $x_0$, but the unstable ones do.

In the case that the unstable directions do not depend on $\tilde{x}$, that is, $E^u(\tilde{x}) = E^u(\tilde{y})$ for any $\tilde{x}, \tilde{y} \in \tilde{M}$ with $x_0 = y_0$, then we say that $f$ is a \emph{special Anosov endomorphism}. Hyperbolic toral endomorphisms are examples of special Anosov endomorphisms, as the unstable direction of each point is given by its unstable eigenspace.

The fact that Anosov endomorphisms that are not invertible or expanding maps are not structurally stable was proven by R. Mañé and C. Pugh \cite{mane1975stability} and F. Przytycki \cite{przytycki1976anosov} in the 1970's, when they introduced the concept of Anosov endomorphisms as we know today. Mañé and Pugh also proved the following proposition, which is very useful to generalize properties of Anosov diffeomorphisms to endomorphisms.

\begin{proposition}[\cite{mane1975stability}]
    \label{prop:F-anosov}
    Let $\overline{M}$ be the universal cover of $M$ and $F: \overline{M} \to \overline{M}$ a lift for $f$. Then $f$ is an Anosov endomorphism if and only if $F: \overline{M} \to \overline{M}$ is an Anosov diffeomorphism. Additionally, the stable bundle of $F$ projects onto that of $f$.
\end{proposition}

Most results for Anosov diffeomorphisms require $M$ to be compact. Even though universal covers are not generally compact, since $F$ is a lift for a map on a compact space, $F$ carries some uniformity, which allows us to prove some results that were originally stated for compact spaces for the lifts proven in Proposition \ref{prop:F-anosov} to be Anosov diffeomorphisms.

\begin{proposition}[\cite{micena-tahzibi2016unstable}]
    \label{prop:F-abs-cont} 
    If $f: M \to M$ is a $C^{1+\alpha}$ Anosov endomorphism, $\alpha > 0$, and $F: \overline{M} \to \overline{M}$ is a lift for $f$ to the universal cover, then there are $W^u_F$ and $W^s_F$ absolutely continuous foliations tangent to $E^u_F$ and $E^s_F$.
\end{proposition}

The above proposition is stated in \cite{micena-tahzibi2016unstable} as Lemma 4.1, and the absolute continuity is defined on Subsection \ref{sec:disin}. The proof is the same as for the compact case, as $F$ projects on the torus, then its derivatives are periodic with respect to compact fundamental domains. With the same argument, we can prove the following just as it is in \cite[\S 19.1]{katok1997introduction}.

\begin{proposition}
    \label{prop:F-C1-hol}
    Let $f: M \to M$ be an Anosov endomorphism and $F: \overline{M} \to \overline{M}$ a lift for $f$ to the universal cover. If the unstable distribution of $F$ has codimension one, then it is $C^1$.
\end{proposition}

In particular, if $\dim{M} = 2$, then we can apply the arguments in Proposition \ref{prop:F-C1-hol} for $F$ and $F^{-1}$, since $F$ is invertible, and both the stable and unstable distributions are $C^1$. This also implies that the stable and unstable holonomies (see Definition \ref{def:hol}) are $C^1$, . In general, these distributions and the associated holonomies are Hölder continuous (see \cite[\S 19.1]{katok1997introduction}). 

\begin{definition}
    \label{def:hol}
     Given a foliation $\mathcal{F}$, we define the \emph{holonomy} $h_{\Sigma_1, \Sigma_2}: \Sigma_1 \to \Sigma_2$ between two local discs $\Sigma_1$ and $\Sigma_2$ transverse to $\mathcal{F}$ by $q \mapsto \mathcal{F}(q) \cap \Sigma_2$, where $\mathcal{F}(q)$ is the leaf of $\mathcal{F}$ containing q.
\end{definition}

That is, a holonomy moves the point $q$ through its leaf on $\mathcal{F}$. For Anosov endomorphisms, we have transverse foliations on the universal cover, so the \emph{stable holonomy} can have local unstable leaves as the discs. When there is no risk of ambiguity, we denote it simply by $h^s$. The same goes for the unstable holonomy.

Another important feature of Anosov endomorphisms on tori is transitivity. The following theorem is a consequence of results in \cite{aoki1994topological} for \emph{topological Anosov maps}, which are continuous surjections with some kind of expansiveness and shadowing property. Anosov endomorphisms are particular cases of topological Anosov maps that are differentiable.

\begin{proposition}[\cite{aoki1994topological}]
    \label{prop:transitive}
    Every Anosov endomorphism on $\mathbb{T}^n$ is transitive.
\end{proposition}

\begin{proof}
    By \cite[Theorem 8.3.5]{aoki1994topological}, every topological Anosov map $f$ on $\mathbb{T}^n$ has its nonwandering set as the whole manifold, that is, $\Omega(f) = \mathbb{T}^n$, which implies transitivity.
\end{proof}

\subsection{Conjugacy}

We say that $A: \mathbb{T}^n \to \mathbb{T}^n$ is the \emph{linearization} of an Anosov endomorphism $f: \mathbb{T}^n \to \mathbb{T}^n$ if $A$ is the unique linear toral endomorphism homotopic to $f$. Much of the behavior of $f$ can be inferred by the one of $A$. In fact, if $f$ is invertible or expansive, $f$ and $A$ are topologically conjugate. In the more general non-invertible setting, this conjugacy does not exist if $f$ is not special, since a conjugacy should preserve stable and unstable manifolds.

The version of Theorem \ref{teo:a} given by F. Micena \cite[Theorem B]{micena2020rigidity} requires the Anosov endomorphism to be \emph{strongly special} --- that is, each point only has one unstable direction and $W_f^s(x)$ is dense for each $x \in M$ ---  in order to guarantee the existence of conjugacy with its linearization and to prove that it is smooth. This relies on Proposition \ref{prop:conj} stated below and given by Aoki and Hiraide in \cite{aoki1994topological} for topological Anosov maps.

\begin{proposition}[\cite{aoki1994topological}]
    \label{prop:conj}
    If $f: \mathbb{T}^n \to \mathbb{T}^n$ is a strongly special Anosov endomorphism, then its linearization $A$ is hyperbolic and $f$ is topologically conjugate to $A$.
\end{proposition}

For Anosov diffeomorphisms, the density of the stable or unstable leaves of each point is equivalent to transitivity. Whether every Anosov diffeomorphism is transitive is still an open question.

However, for general Anosov endomorphisms even in the transitive case, the stable manifolds may be not dense. For example, consider the linear Anosov endomorphism on $\mathbb{T}^3$ induced by the matrix
$$A = \left(\begin{matrix} 
    2 & 1 & 0\\
    1 & 1 & 0\\
    0 & 0 & 2
\end{matrix}\right).$$
It is easy to check that $\dim{E^u}=2$, $\dim{E^s}=1$, $A: \mathbb{T}^3 \to \mathbb{T}^3$ is transitive and $W^u_A(x)$ is dense in $\mathbb{T}^3$ for each $x$, but, if $x = (x_1, x_2, x_3) \in \mathbb{T}^3$, $W^s_A(x)$ is restricted to $\mathbb{T}^2 \times \{x_3\}$, then it is not dense.

In the same year, Naoya Sumi proved that the hypothesis of density on the stable set is not required \cite{sumi1994linearization}. More recently, Moosavi and Tajbakhsh \cite{moosavi2019classification}, very similarly to Aoki and Hiraide and to Sumi, extended this result to topological Anosov maps on nil-manifolds. As a consequence, we have the following.

\begin{proposition}[\cite{sumi1994linearization, moosavi2019classification}]
    An Anosov endomorphism $f: \mathbb{T}^n \to \mathbb{T}^n$ is special if and only if it is conjugate to its linearization by a map $h: \mathbb{T}^n \to \mathbb{T}^n$ homotopic to $Id$.
\end{proposition}

Even a small perturbation of a special Anosov endomorphism may be not special, and a perturbation can, in fact, have uncountable many unstable directions. This is an obstruction to topological conjugacy. On the universal cover, however, we do have a conjugacy.

If $f: \mathbb{T}^n \to \mathbb{T}^n$ is an Anosov endomorphism, by \cite[Proposition 8.2.1]{aoki1994topological} there is a unique continuous surjection $H: \mathbb{R}^n \to \mathbb{R}^n$ on the universal cover with 
    \begin{itemize}
        \item $A \circ H = H \circ F$;
        \item $H$ is uniformly close to $Id$;
        \item $H$ is uniformly continuous.
    \end{itemize}
And, by \cite[Proposition 8.4.2]{aoki1994topological}, $H^{-1}$ exists and it is uniformly continuous, regardless of the distance between $f$ and $A$. One of the key properties to guarantee the invertibility of $H$ is expansiveness. In fact, these two results hold for $f$ topological Anosov map on the $n$-torus.
We can only project $H$ to the torus if $f$ is special. In fact, one can show that even the existence of a semiconjugacy with the linearization on $\mathbb{T}^n$ implies that $f$ is special.

\begin{proposition}
    \label{prop:semiconj-special}
    Let $f: \mathbb{T}^n \to \mathbb{T}^n$ be an Anosov endomorphism and $A$ its linearization. If $A$ is a factor of $f$, then $f$ is special.
\end{proposition}

\begin{proof}
    If $A$ is a factor of $f$, then there is a continuous surjective map $h: \mathbb{T}^n \to \mathbb{T}^n$ such that $h \circ f = A \circ h$. If $H$ is a lift of $h$ to $\mathbb{R}^n$, then $H(x+a) = H(x) + Ba$ for every $x \in \mathbb{R}^n$ and $a \in \mathbb{Z}^n$, where $B: \mathbb{Z}^n \to \mathbb{Z}^n$. Consider $F, A: \mathbb{R}^n \to \mathbb{R}^n$ lifts of $f$ and $A$.

    Given $x \sim y$, i.e., $y = x + a$ with $a \in \mathbb{Z}^n$, we prove that $p(W^u_F(x)) = p(W^u_F(y))$, where $p: \mathbb{R}^2 \to \mathbb{T}^2$ is the canonical projection. Since $H$ takes unstable leaves of $F$ to unstable lines of $A$, then
    $$H(W^u_F(x)) = W^u_A(H(x)) \mbox{ and }$$
    $$H(W^u_F(y)) = H(W^u_F(x+a)) = W^u_A(H(x+a)) = W^u_A(H(x)) + Ba.$$

    Given any $z \in W^u_F(x)$, we have that $H(z) \in W^u_A(H(x))$, then $$H(z) + Ba \in W^u_A(H(x)) + Ba = H(W^u_F(y)).$$ By \cite{aoki1994topological}, $H$ is invertible, then  $H(z) + Ba = H(z+a)$ and, therefore, $z+a \in W^u_F(y)$. Since $z$ is arbitrary, this proves that $p(W^u_F(x)) \subseteq p(W^u_F(y)))$, and the converse is analogous.

    Thus, the set of unstable directions projected from the universal cover for each point in $\mathbb{T}^n$ is unitary. By \cite[Proposition 2.5]{micena-tahzibi2016unstable}, since the set of all unstable directions for a point is the closure of the ones projected from the universal cover, we conclude that $f$ is special.
\end{proof}

In fact, in the conclusion of the previous result, we get that the unstable leaves on the universal cover are invariant under deck transformations if and only if $f$ is special, and the existence of a semiconjugacy on the torus would imply this invariance.

\subsection{Quasi-isometry}

A property frequently required for foliations in $\mathbb{R}^n$ when studying hyperbolic systems is that of quasi-isometry. It means, roughly speaking, that at a large scale the foliation has a uniform length for two points at a given distance.

\begin{definition}
    Given a foliation $\mathcal{W}$ of $\mathbb{R}^n$, with $d_\mathcal{W}$ the distance along the leaves, we say that $\mathcal{W}$ is \emph{quasi-isometric} if there are constants $a, b > 0$ such that, for every $y \in \mathcal{W}(x)$,
    $$d_\mathcal{W}(x,y) \leq a \Vert x-y \Vert +b.$$
\end{definition}

In particular, if the foliation $\mathcal{W}$ is uniformly continuous, the above definition is equivalent to the existence of $Q > 0$ such that, for every $y \in \mathcal{W}(x)$,
    $$d_\mathcal{W}(x,y) \leq Q \Vert x-y \Vert.$$

\begin{proposition}[\cite{micena-tahzibi2016unstable}]
    \label{prop:F-quasi-iso}
    Let $f: \mathbb{T}^n \to \mathbb{T}^n$ be an Anosov endomorphism $C^1$-close to its linearization $A$. Then $W^u_F$ and $W^s_F$ are quasi-isometric.
\end{proposition}

The above proposition is proved in \cite{micena-tahzibi2016unstable} as Lemma 4.4. In the case that $f$ is a special Anosov endomorphism, however, the quasi-isometry is guaranteed by the conjugacy between $f$ and $A$: the stable and unstable foliations lift to foliations of $\mathbb{R}^n$, carrying some uniformity that, together with the fact that $H$ is uniformly bounded and with the global product structure, allows us to bound the lengths properly, as in the invertible setting. Another proof of quasi-isometry of the stable and unstable foliations in the two-dimensional case for $f$ special is given in \cite{he2019dynamical}, using that $W_F^{u/s}$ is homeomorphic to a foliation by lines, therefore it is Reebless and, by the classification of foliations on compact surfaces given in \cite[\S 4.3]{hector1986introduction}, $W_F^{u/s}$ is the suspension of a circle homeomorphism, thus being quasi-isometric.

These arguments do not apply for $W^u_F$ if $f$ is not special, since this foliation does not project to a foliation of $\mathbb{T}^n$, but the quasi-isometry of $W^u_F$ follows from \cite[Proposition 2.10]{hall2019partially} for partially hyperbolic endomorphisms on $\mathbb{T}^2$.

\subsection{Disintegration of probability measures}
\label{sec:disin}

Given $(X, \mathcal{A}, \mu)$ a probability space and $\mathcal{P}$ a partition of $X$ in measurable sets, consider the projection $\pi: X \to \mathcal{P}$ that assigns for each point $x \in X$ the element of $\mathcal{P}$ which contains $x$. Using this projection, one can define a $\sigma$-algebra and a measure in $\mathcal{P}$: $\mathcal{Q} \subseteq \mathcal{P}$ is measurable if $\pi^{-1}(\mathcal{Q})$ is measurable in $X$, and $\hat{\mu}(\mathcal{Q}) = \pi_*\mu(\mathcal{Q}) = \mu(\pi^{-1}(\mathcal{Q}))$.

\begin{definition}
    A family $\{\mu_P\}_{P \in \mathcal{P}}$ of probability measures on $X$ is a \emph{system of conditional measures} (or a \emph{disintegration}) with respect to a partition $\mathcal{P}$ if, for $A \in \mathcal{A}$
    \begin{enumerate}
        \item $P \mapsto \mu_P(A)$ is measurable;
        \item $\mu_P(P) = 1$ for $\hat{\mu}$-almost every $P \in \mathcal{P}$;
        \item $\mu(A) = \int \mu_P(A) d\hat{\mu}(P) $.
    \end{enumerate}
\end{definition}

The conditions on $\mu_P(A)$ for $A$ measurable can be replaced in the above definition by $\int \phi d\mu_P$ for $\phi: X \to \mathbb{R}$ continuous.

If the $\sigma$-algebra is countably generated, given a partition $\mathcal{P}$, if there is a disintegration, it is unique with respect to $\hat{\mu}$. More precisely:

\begin{proposition}
    If the $\sigma$-algebra has a countable generator and $\{\mu_P\}$ and $\{\nu_P\}$ are disintegrations with respect to $\mathcal{P}$, then $\mu_P = \nu_P$ for $\hat{\mu}$-almost every $P \in \mathcal{P}$.
\end{proposition}

The existence of conditional measures is guaranteed by the Rokhlin disintegration theorem for partitions that can be generated by a countable family of sets.

\begin{definition}
    A partition $\mathcal{P}$ is a \emph{measurable partition} (or \emph{countably generated}) with respect to $\mu$ if there is $X_0 \in X$ with $\mu(X_0) = 1$ and a family $\{A_i\}_{i \in \mathbb{N}}$ of measurable sets such that, given $P \in \mathcal{P}$, there is $\{P_i\}_{i \in \mathbb{N}}$ with $P_i \in \{A_i, A_i^\complement\}$ such that $P = \bigcap\limits_{i \in \mathbb{N}} P_i$ restricted to $X_0$.
\end{definition}

\begin{theorem}[Rokhlin disintegration]
    If $X$ is a complete and separable metric space and $\mathcal{P}$ is a measurable partition, then the probability $\mu$ has a disintegration on a family of conditional measures $\mu_P$.
\end{theorem}

For a separable metric space $X$, the Borel $\sigma$-algebra is countably generated, and the disintegration given by Rokhlin disintegration theorem is unique for $\hat{\mu}$-almost every $P$. In particular, if the partition $\mathcal{P}$ is invariant for a measurable function $T: X \to X$ that preserves $\mu$, then $T$ carries conditional measures to conditional measures, that is, $T_*\mu_P = \mu_{T(P)}$ for $\hat{\mu}$-almost every $P$, since $\{T_*\mu_P\}$ is a disintegration of $\mu$ with respect to $\mathcal{P}$.

Returning to our closed Riemannian $n$-manifold $M$, a \emph{foliated box} for a foliation $\mathcal{F}$ of dimension $k$ on $M$ is given by a local leaf $X$ and a local $(n-k)$-dimensional transversal $Y$ to this leaf. The foliated box $\mathcal{B}$ is homeomorphic to $X \times Y$ and the map $\phi: X \times Y \to \mathcal{B}$ takes sets on the form $X \times \{y\}$ to local leaves $\mathcal{F}_{loc}(y)$. We identify $X \times Y$ with $\mathcal{B}$ and $X \times \{y\}$ with $\mathcal{F}_{loc}(y)$.

Note that a foliated box has its local leaves as a measurable partition. Indeed, with the Riemannian structure inherited from $M$, $Y$ is a separable metric space, having a countable base of open sets $\{Y_i\}_{i \in \mathbb{N}}$. The sets $A_i := \bigcup\limits_{y \in Y_i} \mathcal{F}_{loc}(y)$ for $i \in \mathbb{N}$ form a countable generator for the partition, independent of the measure. Then, for foliated boxes, we can always consider a disintegration of any probability measure, by the Rokhlin disintegration theorem.

This allows us to define absolute continuity.

\begin{definition}
    A foliation $\mathcal{F}$ is \emph{(leafwise) absolutely continuous} if, given a foliated box, the conditional volume $m_{\mathcal{F}(x)}$ on each leaf is equivalent to the induced Lebesgue measure $\lambda_{\mathcal{F}(x)}$.
\end{definition}

This means that a set $U \subseteq M$ has zero volume if and only if it has null $m_{\mathcal{F}(x)}$-measure at almost every leaf $\mathcal{F}(x)$. A stronger notion is the following.

\begin{definition}[\cite{micena2013regularity}]
    \label{def:UBD}
    A one-dimensional foliation $\mathcal{F}$ of $M$ has the \emph{uniform bounded density property} (or \emph{UBD property}) if there is $C > 1$ such that, for every foliated box $\mathcal{B}$, the disintegration $\{m^\mathcal{B}_x\}$ of volume normalized to $\mathcal{B}$ satisfies
    \begin{equation}
        \label{eq:ubd-def}
        C^{-1} \leq \dfrac{d m^\mathcal{B}_x}{d \hat{\lambda^\mathcal{B}_x}} \leq C,
    \end{equation}
    where $\hat{\lambda^\mathcal{B}_x}$ is the normalized induced volume in the connected component of $\mathcal{F} \cap \mathcal{B}$ which contains $x$.
\end{definition}

If two measures $\mu$ and $\nu$ defined over the same measurable space are equivalent and their Radon--Nikodym derivative is bounded from above and below as in equation (\ref{eq:ubd-def}), we say that they are \textit{uniformly equivalent}, we use the notation $\mu \stackrel{u}{\sim} \nu$. It is easy to check that it defines an equivalence relation. If two families of measures $\{\mu\}_{i \in \mathcal{I}}$ and $\{\nu\}_{i \in \mathcal{I}}$ are such that $\mu_i \stackrel{u}{\sim} \nu_i$ for each $i \in \mathcal{I}$, and the boundedness constant does not depend on $i$, we say that \textit{$\mu_i \stackrel{u}{\sim} \nu_i$ with uniform constant}.

For an Anosov endomorphism, the unstable foliation may not exist globally, so we say that \emph{the unstable foliation of $f$ has the UBD property} if the unstable foliation of a lift $F: \mathbb{R}^2 \to \mathbb{R}^2$ of $f$ has the UBD property.

\section{Proof of Theorem \ref{teo:a}}
\label{sec:proofa}

This proof is an adaptation of the proof in \cite{gogolev22c} for the non-invertible two-dimensional case. We want to apply Journé's Regularity Theorem \ref{teo:journe}. Even if $f$ and $g$ are not invertible, we can use local transverse foliations on the universal cover to apply Theorem \ref{teo:journe}. So we prove that the lift $H$ of $h$ restricted to the leaves of some local unstable foliation is $C^{k}$. For the stable foliations, the proof is analogous. Therefore, $H$ is $C^{k}$ on a small foliated box on the covering space.

\begin{theorem}[\cite{journe1988regularity}]
    \label{teo:journe}
    Let $M_j$ be a manifold, $W^s_j$, $W^u_j$ continuous transverse foliations with uniformly smooth leaves ($j = 1,2$) and $h: M_1 \to M_2$ a homeomorphism such that $h(W^\sigma_1) = W^\sigma_2$ ($\sigma = s, u$). If $h$ restricted to the leaves of the foliations $W^s_1$ and $W^u_1$ is uniformly $C^{r+\alpha}$, with $r \in \mathbb{N}$ and $\alpha \in (0,1)$, then $h$ is $C^{r+\alpha}$.
\end{theorem}

Let $F, G: \mathbb{R}^2 \to \mathbb{R}^2$ be lifts for $f$ and $g$, and $p: \mathbb{R}^2 \to \mathbb{T}^2$ the canonical projection. $F$ and $G$ are Anosov diffeomorphisms and have stable and unstable foliations $W^s_{F/G}$ and $W^u_{F/G}$, which are quasi-isometric. For any $\xi \in \mathbb{T}^2$, we can consider a foliated box $\mathcal{B}$ with respect to $f$ containing $\xi$ by fixing $\overline{\xi} \in p^{-1}(\xi)$ and projecting a small foliated box $\mathcal{B}_{\overline{\xi}}$ with respect to $W^u_F$ containing $\overline{\xi}$, such that $\restr{p}{\mathcal{B}_{\overline{\xi}}}$ is a bijection over its image $\mathcal{B}$. In particular, $\restr{p}{\mathcal{B}_{\overline{\xi}}}$ is an isometry, which allows us to work either on $\mathcal{B}$ or $\mathcal{B}_{\overline{\xi}}$, and the regularity of $h$ on $\mathcal{B}$ is the same as the regularity of $H$ on $\mathcal{B}_{\overline{\xi}}$. On $\mathcal{B}$, the projected unstable foliation is transverse to $W^s_f$. Consider $H(\mathcal{B})$ the foliated box with respect to $g$ obtained by applying $H$. 

The stable leaves on $\mathcal{B}$ do not depend on the choice of $\overline{\xi}$, and we denote them by $W^s_\mathcal{B}(x)$. The unstable leaves, however, do depend on the choice of $\overline{\xi}$. More precisely, given $x \in \mathcal{B}$, the unstable leaf $W^u_\mathcal{B}(x)$ is a local unstable leaf with respect to the orbit $\tilde{x} = \{p(F^k(\overline{x}))\}_{k \in \mathbb{Z}}$, where $\overline{x} = \restr{p}{\mathcal{B}_{\overline{\xi}}}^{-1}(x)$, and we keep this notation along this proof.

The following classical theorem, which we apply on the next lemma, holds for Anosov endomorphisms, as observed by F. Micena in \cite{micena2020rigidity}.

    \begin{theorem}[Livshitz Theorem]
        \label{teo:Livshitz}
        Let $M$ be a Riemannian manifold and $f: M \to M$ be a transitive Anosov endomorphism. If $\varphi_1, \varphi_2: M \to \mathbb{R}$ are Hölder continuous and
        $$\prod_{i=1}^n \varphi_1(f^i(x)) = \prod_{i=1}^n \varphi_2(f^i(x)) \mbox{ for all $x$ such that } f^n(x) = x  \mbox{, with }n \in \mathbb{N},$$
        then there is a function $P: M \to \mathbb{R}$ such that
        $\dfrac{\varphi_1}{\varphi_2} = \dfrac{P \circ f}{P}.$
        $P$ is Hölder continuous and it is unique up to a multiplicative constant.
    \end{theorem}
    
\begin{lemma}    
    \label{lem:teo2-1}
    $h$ is uniformly Lipschitz along $W^u_f(\tilde{x})$.
\end{lemma}

\begin{proof}

    Fix any $q \in \mathcal{B}$ and consider $h_q: W^u_f(\tilde{q}) \to W^u_g(h(\tilde{q}))$. We aim to show that $h_q$ is Lipschitz with a constant independent of $\mathcal{B}$ and $q$, that is, that there is $K > 0$ such that $d^u_g(h_q(x), h_q(y)) \leq K d^u_f(x,y)$, with $d^u$ the distance along the leaves. We actually do that for equivalent metrics $\tilde{d}_{f/g}$ along the unstable leaves of $f/g$, defined in using $\rho_{f/g}$ as follows.
    
    Given $x \in \mathcal{B}$ and $y \in W^u_\mathcal{B}(x)$, consider
    \begin{equation}
        \label{eq:density}
        \rho_f(x,y) := \prod_{n=1}^\infty \dfrac{D^u_F(F^{-n}(\overline{x}))}{D^u_F(F^{-n}(\overline{y}))},
    \end{equation}
    where $D_F^u(z) = \left| \restr{DF}{E^u_F}(z) \right|$ and $\overline{x}$ and $\overline{y}$ are the lifts of $x$ and $y$ in $\mathcal{B}_{\overline{\xi}}$. 

    Observe that, since $f$ is not necessarily special, each point may have more than one unstable direction for $f$, and each lift of a point on $\mathbb{R}^2$ may have a different unstable direction for $F$. Even so, $D_F^u(\cdot)$ is well defined for points in $\mathcal{B}$, since we fix a lift of this point in $\mathbb{R}^2$ as the one belonging to $\mathcal{B}_{\overline{\xi}}$, and $\rho_f(x, \cdot): W^u_\mathcal{B}(\tilde{x}) \to \mathbb{R}$ is well defined and Hölder continuous, since $D^u_F$ is uniformly bounded and $F^{-n}$ contracts uniformly. Furthermore, given $x \in \mathcal{B}$, $D_F^u(\cdot)$ is well defined for points on $W^u_f(\tilde{x})$, which is the projection of the global unstable leaf of $\overline{x}$ with respect to $F$. Indeed, $p$ is a bijection between $W^u_F(x)$ and $W^u_f(\tilde{x})$, which makes the choice of point on $\mathbb{R}^2$ to compute $D_F^u(\cdot)$ unambiguous. Additionally, $D_F^u(\cdot)$ is well defined for points belonging to iterates of these global unstable leaves, that is, on $p(F^i(W^u_F(\overline{x})))$ for every $i \in \mathbb{Z}$, which are precisely $W^u_f(\tilde{f}^i(\tilde{x}))$ for each $i \in \mathbb{Z}$.

    Thus, if $x \in \mathcal{B}$ and $y \in W^u_f(\tilde{x})$, it makes sense to calculate $\rho_f(f^i(x), f^i(y))$. Then
        \begin{equation}
            \label{eq:rho-A3}
            \rho_f(f(x), f(y)) = \dfrac{D^u_F(\overline{x})}{D^u_F(\overline{y})}\rho_f(x, y).
        \end{equation}
    Additionally, $\rho_f(\cdot, \cdot)$ is the unique continuous function satisfying both (\ref{eq:rho-A3}) and $\rho_f(x, x)=1$. Also, note that, for all $K > 0$, there exists $C > 0$ such that $d^u(x, y) < K$ implies $C^{-1} < \rho_f(x,y) < C$. We define $\rho_g(\cdot, \cdot)$ analogously using the foliated box $H(\mathcal{B})$.
    
    Let $\lambda_q$ be the induced volume on $W^u_f(\tilde{q})$. For $x, y \in W^u_f(\tilde{q})$,

    \begin{equation}
        \label{eq:dist}
        \tilde{d}_f(x,y) := \int_x^y \rho_f(x,z)d\lambda_q(z)
    \end{equation}
    is a metric, and it is equivalent to $d^u_f$, since $\rho_f(x, z)$ is uniformly bounded for $z$ from $x$ to $y$ along the leaf. Moreover, $\tilde{d}_f(f(x),f(y)) = D^u_F(\overline{x})\tilde{d}_f(x,y)$, and, inductively, for all $n \in \mathbb{N}$
    \begin{equation}
        \label{eq:tilde-d-B2}
        \tilde{d}_f(f^n(x),f^n(y)) = \prod\limits_{i=0}^{n-1} D^u_F(F^i(\overline{x})) \; \tilde{d}_f(x,y).
    \end{equation}
    
    We also have that $\tilde{d}_f$ is uniformly continuous: for all $\varepsilon > 0$ there exists $\delta > 0$ such that for all $x, y, z, q$ with $y \in W^u_f(\tilde{x})$, $q \in W^u_f(\tilde{z})$, $z \in B(x, \delta)$, and $q \in B(y, \delta)$ we have that $\vert \tilde{d}_f(x,y) - \tilde{d}_f(z,q) \vert < \varepsilon$.

    Therefore, to prove that $h_q$ is Lipschitz, it suffices to show that $\tilde{d}_g(h_q(x),h_q(y)) < K \tilde{d}_f(x,y)$ for an uniform $K$, where $\tilde{d}_g$ is defined analogously.

    Since $H$ is a lift for the conjugacy $h$, $H(x + m) = H(x) + m$ for all $x \in \mathbb{R}^2$ and $m \in \mathbb{Z}^2$, which implies that there is $C > 0$ such that $\Vert H(x) - H(y) \Vert \leq C \Vert x-y \Vert$ for $\Vert x-y \Vert \geq 1$, where $\Vert . \Vert$ is the Euclidean norm on $\mathbb{R}^2$.

    Using the quasi-isometry of the unstable foliations $W^u_F$ and $W^u_G$, and projecting on $\mathbb{T}^2$, the same inequality is valid for the induced metric on unstable leaves, i.e., there is $C >0$ such that
    \begin{equation}
        \label{eq:lips-distgrande}
        d^u_g(h(x), h(y)) \leq C d^u_f(x, y) \mbox{ for } d^u_f(x,y) \geq 1.
    \end{equation}

    Therefore, we already have the Lipschitz property for points far enough apart. For $x$ and $y$ close, we use the Livshitz Theorem. 

    We apply Theorem \ref{teo:Livshitz} with $\varphi_1(z) = D^u_F(\overline{z})$ and $\varphi_2(z) = D^u_G(H(\overline{z}))$. Note that $\varphi_1(z)$ is only well defined for $z \in \mathcal{B}$, or for $z \in W^u_f(\tilde{f}^i(\tilde{x}))$ for each $i \in \mathbb{Z}$, where $x \in \mathcal{B}$. But the transitivity of $f$ implies that there is a $x \in \mathcal{B}$ with dense orbit. Since $\varphi_1$ is well defined for every point in the orbit of $x$, it can be extended. The same goes for $\varphi_2$, and $\varphi_1$ and $\varphi_2$ satisfy the hypothesis on periodic points due to our hypothesis on Lyapunov exponents. Therefore, it follows from the Livshitz Theorem that
    $$\dfrac{\varphi_1(x)}{\varphi_2(x)} = \dfrac{D^u_F(\overline{x})}{D^u_G(H(\overline{x}))} = \dfrac{P(f(x))}{P(x)},$$
    and, inductively,
    \begin{equation}
        \label{eq:Livshitz}
        \dfrac{P(f^n(x))}{P(x)} = \prod_{i=0}^{n-1} \dfrac{D^u_F(F^i(\overline{x}))}{D^u_G(H(F^i(\overline{x})))} \mbox{ for all } x \in \mathbb{T}^2 \mbox{ and } n \in \mathbb{N}.
    \end{equation}

    Given $x, y \in W^u_f(\tilde{q})$, let $N \in \mathbb{N}$ be the smallest $n$ such that $d^u_f(f^n(x), f^n(y)) \geq 1$. Then $d^u_g(h(f^n(x)), h(f^n(y))) \leq C d^u_f(f^n(x), f^n(y))$.

    By the property (\ref{eq:tilde-d-B2}) of the distance $\tilde{d}_f$ we have that
    $$\tilde{d}_f(x,y) = \dfrac{\tilde{d}_f(f^N(x), f^N(y))}{\prod\limits_{i=0}^{N-1} D^u_F(F^i(\overline{x}))},$$
    and an analogous equality is valid for $\tilde{d}_g$, therefore

    $$\dfrac{\tilde{d}_g(h(x),h(y))}{\tilde{d}_f(x,y)} = \prod\limits_{i=0}^{N-1} \dfrac{D^u_F(F^i(\overline{x}))}{D^u_G(H(F^i(\overline{x})))} \dfrac{\tilde{d}_g(h(f^N(x)), h(f^N(y)))}{\tilde{d}_f(f^N(x), f^N(y))}.$$

    The first term of this product is equal to $\dfrac{P(f^n(x))}{P(x)}$, which is bounded, and the second one is bounded by the Lipschitz constant for distant points given by the inequality (\ref{eq:lips-distgrande}) but for the equivalent metric $\tilde{d}_f$.
\end{proof}

Since $h$ is Lipschitz along $W^u_f(\tilde{x})$ for $x$ in $\mathcal{B}$, then $h$ is \emph{u-differentiable} (differentiable along unstable leaves) for almost every point with respect to the Lebesgue measure induced on the leaves. If $h$ is u-differentiable for $x$, then it is for $f^k(x)$, $k \in \mathbb{N}$, since $f$ is $C^k$ and the unstable leaves are $f$-invariant, and the same goes for $H$ and $F^k(\overline{x})$, $k \in \mathbb{Z}$.

\begin{lemma}
    \label{lem:teo2-2}
    If $h$ is $u$-differentiable at $x \in W^u_B(q)$, then it is u-differentiable at every $y \in W^s_B(x)$.
\end{lemma}

\begin{proof}
    This proof uses the same idea as Step 1 on Lemma 5 in \cite{gogolev22c}, which is to estimate the derivative of a point using a nearby u-differentiable point and the unstable leaves. 

    Given $y \in W^s_\mathcal{B}(x)$, for each $n \in \mathbb{N}$ we fix a $y_n \in W^u_F(F^n(\overline{y}))$ close to $F^n(\overline{y})$. We have that $\tilde{d}_f$ and $\tilde{d}_g$ are uniformly continuous, $P$ is Hölder continuous and $H$ is Lipschitz. Then for each small $\varepsilon > 0$ there is $\delta >0$ independent of $n$ such that $\zeta \in B(F^n(\overline{y}), \delta)$ implies
    \begin{equation}
        \label{eq:P-a}
        \vert P(p(\zeta)) - P(f^n(y)) \vert < \varepsilon
    \end{equation}
    and
    there is $w \in W^u_F(\zeta)$ such that $\tilde{d}_f(\zeta,w) = \tilde{d}_f(F^n(\overline{y}), y_n)$, $w$ has the same orientation as $y_n$, $w$ belongs to a small neighborhood of $y_n$ and
    \begin{equation}
        \label{eq:P-b}
        \vert \tilde{d}_g(H(F^n(\overline{y})),H(y_n)) - \tilde{d}_g(H(\zeta),H(w)) \vert < \varepsilon.
    \end{equation}

    \begin{figure}[ht]
        \centering
        \def\svgwidth{.7\linewidth}
        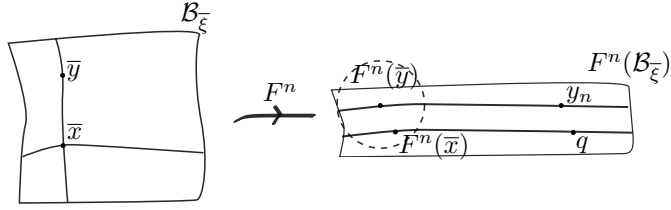
        \caption{We estimate the u-derivative of $H$ at $F^n(\overline{y})$ using the one of $H$ at $F^n(\overline{x})$.}
        \label{fig:lemma3}
    \end{figure}

    Since $y \in W^s_\mathcal{B}(x)$, $d(F^k(\overline{y}), F^k(\overline{x})) \xrightarrow{k \to \infty} 0$ and we can fix $n \in \mathbb{N}$ such that $F^n(\overline{x}) \in B(F^n(\overline{y}), \delta)$. Then there is $w \in W^u_F(F^n(\overline{x}))$ that satisfies the inequality (\ref{eq:P-b}) above (see Figure \ref{fig:lemma3}). Let $z \in W^u_F(\overline{x})$ be such that $F^n(z) = w$. Then, taking $n$ sufficiently large, we have $\tilde{d}_f(\overline{x}, z)$ small enough so that
    \begin{equation}
        \label{eq:derivada}
        \left\vert \dfrac{\tilde{d}_g(H(\overline{x}), H(z))}{\tilde{d}_f(\overline{x}, z)} - D^u_H(\overline{x})\right\vert < \varepsilon.
    \end{equation}

    Then
    \begin{align*}
        \tilde{d}_g(H(F^N(\overline{y})),H(y_N)) &\stackrel{(\ref{eq:P-b})}{=} \varepsilon_1 + \tilde{d}_g(H(F^n(\overline{x})), H(F^n(z)))\\
        &\stackrel{(\ref{eq:tilde-d-B2})}{=} \varepsilon_1 + \prod\limits_{i=0}^{N-1} D^u_G(H(F^i(\overline{x}))) \; \tilde{d}_g(H(\overline{x}),H(z))\\
        & \stackrel{(\ref{eq:derivada})}{=} \varepsilon_1 + \prod\limits_{i=0}^{N-1} D^u_G(H(F^i(\overline{x}))) \; (D^u_H(\overline{x}) + \varepsilon_2)\tilde{d}_f(\overline{x}, z)\\
        & \stackrel{(\ref{eq:tilde-d-B2})}{=} \varepsilon_1 + \prod\limits_{i=0}^{N-1} D^u_G(H(F^i(\overline{x}))) \; (D^u_H(\overline{x}) + \varepsilon_2) \dfrac{\tilde{d}_f(F^N(\overline{x}), F^N(z))}{\prod\limits_{i=0}^{N-1} D^u_F(F^i(\overline{x}))}\\
        & \stackrel[(\ref{eq:Livshitz})]{(\ref{eq:P-b})}{=} \varepsilon_1 + \dfrac{P(x)}{P(f^N(x))} \; (D^u_H(\overline{x}) + \varepsilon_2) \tilde{d}_f(F^n(\overline{y}), y_n)\\
        &\stackrel{(\ref{eq:P-a})}{=} \varepsilon_1 + \dfrac{P(x)}{P(f^N(y)) + \varepsilon_3} \; (D^u_H(\overline{x}) + \varepsilon_2) \tilde{d}_f(F^n(\overline{y}), y_n),
    \end{align*}
    with $\vert \varepsilon_i \vert < \varepsilon$, $i = 1, 2, 3$. As $\varepsilon \longrightarrow 0$, we have
    $$\dfrac{\tilde{d}_g(H(F^N(\overline{y})),H(y_N))}{\tilde{d}_f(F^n(\overline{y}), y_N)} = \dfrac{P(x)}{P(f^N(y))} D^u_H(\overline{x}),$$
    with the right-hand side not depending on $y_N$, which implies that $H$ is u-differentiable at $F^n(\overline{y})$. Therefore, $H$ is u-differentiable at $\overline{y}$ and $h$ is u-differentiable at $y$.
\end{proof}

Consider $\mathcal{K} := \{x \in \mathcal{B}: h \mbox{ is u-differentiable at } x \}$ and $$\mathcal{K}(q) := \{x \in W^u_\mathcal{B}(q): h \mbox{ is u-differentiable at } x \}.$$ Lemma \ref{lem:teo2-2} implies that $\bigcup\limits_{x \in \mathcal{K}(q)} W^s_\mathcal{B}(x) \subseteq \mathcal{K}$. But $W^u_\mathcal{B}(q)$ is transverse to the foliation $W^s_\mathcal{B}$, which is an absolutely continuous foliation of $\mathcal{B}$. Then, since $\mathcal{K}(q) \subseteq W^u_\mathcal{B}(q)$ has full measure, $\lambda_\mathcal{B}(\mathcal{K}) = 1$, where $\lambda_\mathcal{B}$ is the normalized Lebesgue measure on $\mathcal{B}$. This implies that $\mathcal{K}$ is dense in $\mathcal{B}$.

Note in the proof of Lemma \ref{lem:teo2-2} that we did not use the fact that $y \in W^s_\mathcal{B}(x)$, unless to make sure that there is a u-differentiable point sufficiently close to the orbit of $y$. Since the u-differentiable points make a dense set, we can estimate the u-derivative of any point using nearby u-differentiable points, as was done in the lemma. In particular, if $x \in \mathcal{B}$ is a u-differentiable point for $h$ sufficiently close to $y$, then
$$D^u_H(\overline{y}) = \dfrac{P(x)}{P(y)}D^u_H(\overline{x}).$$

Therefore, $\mathcal{K} = \mathcal{B}$, $h$ is u-differentiable for each point in $\mathcal{B}$ and $D^u_H$ is $C^{1+\alpha}$.

In order to obtain for $h$ the same regularity of $f$, let us see that 
\begin{equation}
    \label{eq:rho}
    \rho_g(h(x),h(y)) = \dfrac{D^u_H(\overline{x})}{D^u_H(\overline{y})}\rho_f(x,y),    
\end{equation}
for $x, y \in W^u_\mathcal{B}(q)$. Consider

$$\tilde{\rho}_g(h(x),h(y)) := \dfrac{D^u_H(\overline{x})}{D^u_H(\overline{y})}\rho_f(x,y).$$

$\tilde{\rho}_g$ satisfies (\ref{eq:rho-A3}) for g, therefore it is equal to $\rho_g$ by uniqueness. Indeed,
\begin{align*}
    \tilde{\rho}_g(g(h(x)),g(h(y))) & = \tilde{\rho}_g(h(f(x)),h(f(y))) = \dfrac{D^u_H(F(\overline{x}))}{D^u_H(F(\overline{y}))}\rho_f(f(x),f(y))\\
    &= \dfrac{D^u_H(F(\overline{x}))}{D^u_H(F(\overline{y}))}\dfrac{D^u_F(\overline{x})}{D^u_F(\overline{y})}\rho_f(x, y) = \dfrac{D^u_{H \circ F}(\overline{x})}{D^u_{H \circ F}(\overline{y})}\rho_f(x, y)\\
    &= \dfrac{D^u_{G \circ H}(\overline{x})}{D^u_{G \circ H}(\overline{y})}\rho_f(x, y) = \dfrac{D^u_G(H(\overline{x}))}{D^u_G(H(\overline{y}))} \dfrac{D^u_{H}(\overline{x})}{D^u_{H}(\overline{y})}\rho_f(x, y)\\
    &= \dfrac{D^u_G(H(\overline{x}))}{D^u_G(H(\overline{y}))}  \tilde{\rho}_g(h(x),h(y)).
\end{align*}

Now, we apply the same argument as \cite[Lemma 2.4]{gogolev2017bootstrap}: by \cite[Lemma 3.7]{micena2020rigidity}, $\rho_g$ and $\rho_f$ are $C^{k-1}$. Therefore, the relation (\ref{eq:rho}) implies that $D^u_H$ is $C^{k-1}$, and $H$ is $C^k$ along $W^u_F$.

The proof for the stable direction is analogous, since we lift the points to the fixed foliated box $\mathcal{B}_{\overline{\xi}}$ to calculate $\rho$ using $F$, which is invertible, we can define $\rho^s_f(x,y)$ for the stable direction analogously by exchanging $F$ with $F^{-1}$. Then we can apply Theorem \ref{teo:journe} for $H$ in $\mathcal{B}_{\overline{\xi}}$ and conclude that $H$ is $C^{k}$.

\section{Proof of Theorem \ref{teo:b}}
\label{sec:proofb}

In this Section, we first prove a simpler formulation of our result, in the form of Theorem \ref{teo:b}. In Subsection \ref{subsec:conserv} we address how the preservation of volume is different in our setting when compared to the invertible case addressed in \cite{varao2018rigidity}. In Subsection \ref{sec:teoc} we explain how we can expand Theorem \ref{teo:b} for a more general setting and the changes in the proof to do so.

\subsection{Proof of Theorem \ref{teo:b}}
\label{subsec:teob}
    
Let $F: \mathbb{R}^2 \to \mathbb{R}^2$ be a lift of $f$ to the universal cover, and $p: x \mapsto [x]$ the canonical projection from $\mathbb{R}^2$ to $\mathbb{T}^2$. Let $H: \mathbb{R}^2 \to \mathbb{R}^2$ be the conjugacy between $F$ and $A$. This conjugacy implies that the stable and unstable foliations of $F$ have global product structure.

We already know that the unstable and stable foliations $W^u_F$ and $W^s_F$ are absolutely continuous by Proposition \ref{prop:F-abs-cont}. So, by requiring the UBD property (promoting the absolute continuity to a uniform formulation) for the unstable foliation, we establish on this theorem that the exponents are constant at each point on $\mathbb{T}^2$.

$F$ has exactly one fixed point by \cite{franks1969anosov}, and we can suppose without loss of generality that $F(0) = 0$. Let $B := W_F^s(0)$ be the stable leaf of $0$ with respect to $F$. Then $B$ is $F$-invariant and the unstable leaves of $F$ intersect $B$ transversely at a unique point. So we can define $p^u_F: \mathbb{R}^2 \to B$ as the projection that takes each point $z$ to $W_F^u(z) \cap B$. We can also define an orientation on each leaf by choosing a ``side'' of $B$ as positive. Assuming that $F$ preserves the orientation on unstable manifolds, or working with $F^2$ instead, consider the \emph{foliated strip}
$$\mathcal{B}_0 := \{y \in \mathbb{R}^2 : d^u(p^u_F(y), y) \leq \delta_0, \; y \in W^{u,+}_F \},$$
where $\delta_0 > 0$ is a constant such that $p(\mathcal{B}_0) = \mathbb{T}^2$. $\mathcal{B}_0$ can be seen as a strip ``above'' $B$ whose projection covers the whole torus. Let $\mathcal{B}_k := F^k(\mathcal{B}_0)$ be the iterates of $\mathcal{B}_0$. Since $B$ is $F$-invariant and it expands along the unstable leaves, $\mathcal{B}_{k-1} \subsetneq \mathcal{B}_{k}$.

Let $m^k := \restr{m}{\mathcal{B}_{k}}$ be the induced volume on the foliated strip $\mathcal{B}_{k}$. Even with the volume of $\mathcal{B}_{k}$ being infinite, we can still consider conditional probability measures on foliated boxes in $\mathcal{B}_{k}$. Indeed, $B$ has a countable base of open sets $\{Y_i\}_{i \in \mathbb{N}}$, and $\mathcal{B}_{k}$ can be divided in sets in the form $(\restr{p^u_F}{\mathcal{B}_{k}})^{-1}(Y_i)$ that are foliated boxes in which we can disintegrate the finite volume. So we can consider $m^{k}_x$ the conditional probability measure in $W_k(x) := W^u_F(x) \cap \mathcal{B}_{k}$ for $m$-almost every point in $\mathcal{B}_k$. Furthermore, this probability measure is unique for almost every leaf.

We immediately have that, since the Jacobian of $f$ is constant, these measures behave nicely under $F$.

\begin{lemma}
    \label{lem:claim1}
    $\dfrac{dm^k_x}{d F^k_*m^0_{F^{-k}(x)}} = 1$.
\end{lemma}
    
    \begin{proof} We have that $Jf$ is constant, and it satisfies $Jf \equiv \sigma = \vert \det A \vert$. Considering that

\begin{itemize}
    \item[a)] $dF^k_* m^0 = \underbrace{\vert \det DF^{-k}(\cdot) \vert}_{JF^{-k}(\cdot)} dm^k$ in $\mathcal{B}_k$;
    \item[b)] $\{m^k_x\}_x$ is a disintegration of $m^k$; 
    \item[c)] $\{F^k_* m_{F^{-k}(x)}^0\}_x$ is a disintegration of $F^k_* m^0$.
\end{itemize}

Then, for any $\varphi: \mathcal{B}_k \to \mathbb{R}$, we have on one hand that
\begin{equation}
    \label{eq:int-phi1}
    \int_{\mathcal{B}_k} \varphi dF^k_* m^0 \stackrel{c)}{=} \int_B \left( \int_{W_k(x)} \varphi(z) dF^k_* m_{F^{-k}(x)}^0(z)  \right) dF^k_* \mu^0(x).
\end{equation}

On the other hand, we have that
\begin{equation}
    \label{eq:int-phi2}
    \int_{\mathcal{B}_k} \varphi dF^k_* m^0 \stackrel{a)}{=} \int_{\mathcal{B}_k} \varphi \; JF^{-k} dm^k \stackrel{b)}{=} 
    \int_B \left( \int_{W_k(x)} \varphi(z) JF^{-k}(z) dm^k_x(z) \right) d\mu^k(x),
\end{equation}
where $\mu_k$ is the transversal measure induced by the partition of $\mathcal{B}_k$ into its leaves. By comparing (\ref{eq:int-phi1}) and (\ref{eq:int-phi2}), we have that
\begin{equation}
    \label{eq:compara-dm}
    dF^k_* m_{F^{-k}(x)}^0(z) \dfrac{dF^k_* \mu^0}{d\mu^k}(x) = JF^{-k}(z) dm^k_x(z).
\end{equation}

The transverse measures $\mu_k$ and $F^k_* \mu^0$ given by the disintegration satisfy
$$\dfrac{dF^k_* \mu^0}{d\mu^k}(x) = \lim_{\varepsilon \to 0} \dfrac{F^k_* \mu^0(I^B_\varepsilon)}{\mu_k(I^B_\varepsilon)} = \lim_{\varepsilon \to 0} \dfrac{m \circ F^{-k}(A_\varepsilon)}{m(A_\varepsilon)},$$
where $I^B_\varepsilon \subseteq B$ is a ball with center $x$ and radius $\varepsilon$ on $B = W^s_F(0)$, and $A_\varepsilon = \bigcup_{y \in I^B_\varepsilon} W_k(y)$. Thus, $JF$ being constant implies that $$\dfrac{dF^k_* \mu^0}{d\mu^k}(x) = \sigma^{-k},$$ which cancels with $JF^{-k}(z)$ on equation (\ref{eq:compara-dm}).
        
    \end{proof}  

The idea of this proof is to construct measures with respect to $F$ on these local leaves of $\mathcal{B}_k$ with densities that decrease as $k$ grows with a rate equal to the unstable Lyapunov exponent of the linearization $A$. This allows us to conclude that the unstable Lyapunov exponents of $F$ and $A$ are the same. More precisely, we want to construct measures $\eta_x$ uniformly equivalent to the volume induced on the global leaves such that $F_*\eta_x = \alpha^{-1} \eta_{F(x)}$, with $\alpha$ being the unstable eigenvalue of $A$.

Consider the measures $\eta_x^k$ defined inductively as

$$\eta_x^0 :=m_x^0 \mbox{ and}$$
$$\eta_x^k := \alpha F_* \eta_{F^{-1}(x)}^{k-1} = \alpha^k F^k_* m_{F^{-k}(x)}^0,$$
where $\alpha$ is the unstable eigenvalue of $A$. Since $\vert \alpha \vert > 1$, $\eta_x^k$ is not a probability. Actually we will see that it is comparable with $\lambda_x$, where $\lambda_x$ is the induced volume along the global leaf $W^u_F(x)$. For this, we use the quasi-isometry to show that the length of $W_k(x)$ grows uniformly with $\alpha$ as we apply $F$.

\begin{lemma}
    \label{lem:lemma1} 
    There is $K > 1$ such that, for each $k \in \mathbb{N}$ and $x \in \mathcal{B}_0$, we have that
    $$K^{-1} \leq \dfrac{\lambda_{F^k(x)}(W_k(F^k(x)))}{\alpha^k} \leq K.$$
\end{lemma}

\begin{proof}
    Due to the quasi-isometry, to estimate $\lambda_{F^k(x)}(W_k(F^k(x)))$ it suffices to estimate $\Vert a_k - b_k \Vert$, where $a_k, b_k$ are the extreme points of $W_k(F^k(x))$. But $a_k = F^k(a_0)$ and $b_k = F^k(b_0)$, with $a_0$ and $b_0$ being the extreme points of $W_0(x)$. Then
    \begin{align*}
        \Vert F^k(a_0) - F^k(b_0) \Vert & \leq \Vert F^k(a_0) - H \circ F^k(a_0) \Vert + \Vert H \circ F^k(a_0) - H \circ F^k(b_0) \Vert\\
        &+ \Vert H \circ F^k(b_0) - F^k(b_0)\Vert \leq 2\delta + \Vert A^k \circ H(a_0) - A^k \circ H(b_0) \Vert,
    \end{align*}
    where $H(a_0), H(b_0)$ are on the same unstable line for $A$, and $\delta = d(Id, H) > 0$.
    
    Since
    \begin{align*}
        \Vert A^k \circ H(a_0) - A^k \circ H(b_0) \Vert &= \alpha^k \Vert H(a_0) - H(b_0) \Vert\\
        &\leq \alpha^k (\Vert a_0 - b_0 \Vert +2 \delta) \leq \alpha^k (\lambda_x(W_0(x)) + 2\delta),
    \end{align*}
    the upper bound follows because $\lambda_x(W_0(x)) = \gamma_0$, and the lower bound is analogous.
\end{proof}

By the multiplicative ergodic theorem for endomorphisms (see e. g. \cite{qian2009smooth}), the unstable Lyapunov exponent of $f$ on a given point does not depend on the orbit that defines the unstable direction, so we can obtain it at every point in $\mathcal{B}_0$, since $p(\mathcal{B}_0) = \mathbb{T}^2$, $p$ is a local isometry and the projection of unstable leaves of $F$ are unstable leaves of $f$.

The following calculations and estimations are for points $x \in \mathcal{B}_0$ such that $m^k_x$ is well defined for all $k \geq 0$. We would like to construct the measures $\eta_x$ as the limit of measures $\eta^k_x$ for a given $x \in \mathcal{B}_0$. For that, we need $\eta^k_x$ to be well defined for every large enough $k$, which will only be possible for a full volume set on the foliated strip $\mathcal{B}_0$. Indeed, for every $k \geq 0$, there is a full volume set $A_k \subseteq \mathcal{B}_k$ such that $m^k_x$ is well defined for each $x \in A_k$. Consider $A := \bigcap_{k \geq 0} A_k$ and $D := \bigcap_{n \in \mathbb{N}} F^n(A)$. $D$ is $F$-invariant, has full measure in $\mathcal{B}_0$ and $m^k_x$ is well defined for each $x \in D$ and $k \geq 0$.

Now, we construct $\eta_x$ for $x \in D$, and later, using the density of $D$, we construct other measures to compute the Lyapunov exponents for each point in $\mathcal{B}_0$.

\begin{lemma}
    \label{lem:lemma2}
    For m-almost every $x \in \mathcal{B}_0$ there is a measure $\eta_x$ on $W^u_F(x)$ such that $F_* \eta_x = \alpha^{-1} \eta_{F(x)}$ and $\eta_x = \rho_x \lambda_x$ with $\rho_x$ uniformly bounded.
\end{lemma}

\begin{proof}
    By definition, $F_* \eta^k_x = \alpha^{-1} \eta^{k+1}_{F(x)}$. Hence, it suffices to show that the sequence $\{\eta^k_x\}_{k \geq 0}$ has an accumulation point $\eta_x$ such that $\eta^{k+1}_{F(x)}$ converges under the same subsequence to a measure $\eta_{F(x)}$. Firstly, we show that $\eta^k_x$ is uniformly equivalent to the induced volume on the leaf; this will guarantee the existence of accumulation points to $\{\eta^k_x\}_{k \geq 0}$ and the uniformity of $\rho_x$ follows.
    
    Essentially, we have uniform equivalence with uniform constants (not depending on $x$ or $k$) between the following measures
    \begin{enumerate}
        \item $m^k_x \stackrel{u}{\sim} \hat{\lambda}^k_x$ by the UBD property;
        \item $m^k_x \stackrel{u}{\sim} F^k_*m^0_{F^{-k}(x)}$ with uniform constant $1$, since $Jf$ is constant, by Lemma \ref{lem:claim1};
        \item $F^k_*m^0_{F^{-k}(x)} \stackrel{u}{\sim} \hat{\lambda}^k_x$ by items 1 and 2;
        \item $\eta^k_x = \alpha^kF^k_*m^0_{F^{-k}(x)} \stackrel{u}{\sim} \alpha^k \hat{\lambda}^k_x$.
    \end{enumerate}
    By Lemma \ref{lem:lemma1}, $\alpha^k \hat{\lambda}^k_x$ is uniformly equivalent to $\lambda_x$, the induced volume on the global leaf $W^u_F(x)$. Let us now formalize these ideas.
    
    For all $k \in \mathbb{N}$, since the unstable foliation of $F$ has the UBD property, then there is $C > 1$ such that
    \begin{equation}
    \label{eq:UBD}
       C^{-1} \leq \dfrac{dm^k_x}{d \hat{\lambda}^k_x} \leq C.
    \end{equation}

    Using Lemma \ref{lem:claim1}, we multiply the above inequality by $1 \equiv \dfrac{d F^k_*m^0_{F^{-k}(x)}}{dm^k_x}$, and
    
    $$C^{-1} \leq \dfrac{d F^k_*m^0_{F^{-k}(x)}}{d \hat{\lambda}^k_x} \leq C.$$
    By multiplying by $\alpha^k$, we have that
    
    $$C^{-1} \alpha^k d \hat{\lambda}^k_x \leq d \eta^k_x \leq C \alpha^k d \hat{\lambda}^k_x .$$
    
    By Lemma \ref{lem:lemma1},
    $$K^{-1} d \lambda_x \leq \alpha^k d \hat{\lambda}^k_x = \alpha^k \dfrac{d \lambda_x}{\lambda(W_K(x))} \leq K d \lambda_x .$$
    This implies that the densities $\{\eta^k_x\}_{k >0}$ are bounded with
    \begin{equation}
        \label{eq:bound-eta}
        C^{-1}K^{-1}d \lambda_x \leq d \eta^k_x \leq CK d \lambda_x,
    \end{equation}
    then they have an accumulation point. Let $k_i$ index a subsequence such that $\eta_x := \lim_i \eta^{k_i}_x$. The sequence $\{\eta^{k_i+1}_{F(x)}\}_i$ has an accumulation point for a subsequence with indices that also make $\{\eta^k_x\}_{k>j}$ converge. By a diagonal argument, we get $\eta_x$ as desired.
\end{proof}

In order to obtain the Lyapunov exponents, we would like to have the measures $\eta_x$ defined at every point in $\mathcal{B}_0$, but we only have them on a full volume set. To overcome this, given $z \in \mathcal{B}_0$, we define measures $\mathfrak{m}_z$ and $\mathfrak{M}_z$ using the stable holonomy $h^s$ that carries points from a local unstable leaf of $F$ to other by traveling on stable leaves. Let $(z_n)_n$ be a sequence of points in $\mathcal{B}_0$ such that $\eta_{z_n}$ exists and $z_n \xrightarrow{n \to \infty} z$. Let $\mathcal{I}_z$ be the set of connected intervals on $W^u_F(z)$. Define for all $I \in \mathcal{I}_z$
$$\mathfrak{m}_z(I) := \liminf_n \dfrac{1}{n} \sum_{i=0}^{n-1} \eta_{z_i}(h^s(I)),$$
$$\mathfrak{M}_z(I) := \limsup_n \dfrac{1}{n} \sum_{i=0}^{n-1} \eta_{z_i}(h^s(I)).$$

Let us check that these measures are uniformly equivalent to the volume induced on the global leaves and are well behaved with respect to $F$, as are the measures $\eta_x$. There is a constant $\gamma$ independent of the choice of $z$ and $(z_n)_n$ such that

\begin{equation}
    \label{eq:lemma4}
    \gamma^{-1}\lambda_z(I) \leq \mathfrak{m}_z(I) \leq \mathfrak{M}_z(I) \leq \gamma\lambda_z(I),
\end{equation}
for all $I \in \mathcal{I}_z$ small enough. Indeed, by Lemma \ref{lem:lemma2},
$$\mathfrak{m}_z(I) = \liminf_n \dfrac{1}{n} \sum_{i=0}^{n-1} \rho_{z_i}\lambda_{z_i}(h^s(I)) \geq C^{-2}K^{-1}\liminf_n \dfrac{1}{n} \sum_{i=0}^{n-1} \lambda_{z_i}(h^s(I)).$$

By Proposition \ref{prop:F-C1-hol}, holonomies are $C^1$, and, for $I$ small enough, $\lambda_{z_i}(h^s(I))$ is uniformly close to $\lambda_z(I)$. Then $\mathfrak{m}_z(I) \geq \gamma^{-1}\lambda_z(I)$. $\mathfrak{M}_z(I) \leq \gamma\lambda_z(I)$ follows analogously.

For all $I \in \mathcal{I}_{F^k(z)}$ small enough
\begin{equation}
    \label{eq:lemma5}
    \begin{aligned}
        F_*^k\mathfrak{m}_z(I) &= \alpha^{-k}\mathfrak{m}_{F^k(z)}(I),\\
        F_*^k\mathfrak{M}_z(I) &= \alpha^{-k}\mathfrak{M}_{F^k(z)}(I).
    \end{aligned}
\end{equation}
Indeed, since the holonomies are $F$-invariant and by Lemma \ref{lem:lemma2}:
\begin{align*}
    F_*^k\mathfrak{m}_z(I) &= \mathfrak{m}_z(F^{-k}(I)) = \liminf_n \dfrac{1}{n} \sum_{i=0}^{n-1} \eta_{z_i}(h^s(F^{-k}(I))) = \liminf_n \dfrac{1}{n} \sum_{i=0}^{n-1} \eta_{z_i}(F^{-k}(h^s(I)))\\
    &= \liminf_n \dfrac{1}{n} \sum_{i=0}^{n-1} \alpha^{-k}\eta_{F^k(z_i)}(h^s(I)) = \alpha^{-k}\mathfrak{m}_{F^k(z)}(I).
\end{align*}
And for $\mathfrak{M}$ it is analogous.

The relations (\ref{eq:lemma4}) and (\ref{eq:lemma5}) imply that the unstable Lyapunov exponent of $F$ is $\log(\alpha)$. Indeed, for $n \in \mathbb{N}$
$$F_*^n(\gamma^{-1}\lambda_z)(I) \stackrel{(\ref{eq:lemma4})}{\leq} F_*^n(\mathfrak{m}_z)(I) \stackrel{(\ref{eq:lemma5})}{=} \alpha^{-n}\mathfrak{m}_{F^n(z)}(I) \mbox{ and}$$
$$F_*^n(\gamma\lambda_z)(I) \stackrel{(\ref{eq:lemma4})}{\geq} F_*^n(\mathfrak{M}_z)(I) \stackrel{(\ref{eq:lemma5})}{=} \alpha^{-n}\mathfrak{M}_{F^n(z)}(I),$$
therefore

\begin{equation}
    \label{eq:ineq2}
    \gamma^{-1} F_*^n(\lambda_z)(I) \leq \alpha^{-n}\mathfrak{m}_{F^n(z)}(I) \leq \alpha^{-n}\mathfrak{M}_{F^n(z)}(I) \leq \gamma F_*^n(\lambda_z)(I).
\end{equation}

Dividing by $F_*^n(\lambda_z)(I)$ and shrinking the interval $I$, we get the derivative of $F^{-n}$ at the unstable direction. Indeed, if $x \in I$, let $I_\varepsilon := \{y \in W^u_F(F^n(z)): d^u(x,y) < \varepsilon \}$ be the open ball around $x$ on the unstable leaf of $F^n(z)$. Then, since by (\ref{eq:lemma4}) $\mathfrak{m}_{F^n(z)}$ is uniformly bounded with respect to the Lebesgue measure, $\mathfrak{m}_{F^n(z)} = \tau_{F^n(z)}\lambda_{F^n(z)}$ with $\tau_{F^n(z)} \in [\beta^{-1}, \beta]$, so

\begin{align*}
    &\alpha^{-n}\dfrac{\int_{I_\varepsilon}\beta^{-1} d \lambda_{F^n(z)}(\xi)}{\lambda_z(F^{-n}(I_\varepsilon))} \leq \alpha^{-n}\dfrac{\mathfrak{m}_{F^n(z)}(I_\varepsilon)}{F_*^n(\lambda_z)(I_\varepsilon)}\\
    &= \alpha^{-n}\dfrac{\int_{I_\varepsilon}\tau_{F^n(z)}(\xi) d \lambda_{F^n(z)}(\xi)}{\lambda_z(F^{-n}(I_\varepsilon))} \leq \alpha^{-n}\dfrac{\int_{I_\varepsilon}\beta d \lambda_{F^n(z)}(\xi)}{\lambda_z(F^{-n}(I_\varepsilon))},    
\end{align*}
and taking the limit as $\varepsilon \rightarrow 0$

\begin{equation*}
    \alpha^{-n}\beta^{-1}\left\Vert \restr{DF^{-n}}{E^u_F}(x) \right\Vert^{-1} \leq \lim_{\varepsilon \rightarrow 0} \dfrac{\alpha^{-n} \mathfrak{m}_{F^n(z)}(I_\varepsilon)}{\lambda_z(F^{-n}(I_\varepsilon))} \leq \alpha^{-n}\beta \left\Vert \restr{DF^{-n}}{E^u_F}(x) \right\Vert^{-1}.
\end{equation*}

Now by applying $\dfrac{1}{n} \log$ and taking the limit as $n \rightarrow \infty$, we have
$$-\log(\alpha) - \lambda^u_{F^{-1}}(x) \leq 0 \leq -\log(\alpha) - \lambda^u_{F^{-1}}(x),$$
where in the central expression we use the inequality (\ref{eq:ineq2}) to bound $\alpha^{-n} \dfrac{\mathfrak{m}_{F^n(z)}(I_\varepsilon)}{\lambda_z(F^{-n}(I_\varepsilon))}$.

Therefore, $\lambda_{F}^u(x) \equiv  \lambda_{A}^u$ for all $x \in \mathcal{B}_0$. Since $p: \mathbb{R}^2 \to \mathbb{T}^2$ is a local isometry, $\lambda_{f}^u(x) \equiv  \lambda_{A}^u$ for all $x \in \mathbb{T}^2$. By the definition of Lyapunov exponent and the Birkhoff ergodic theorem,

$$\lambda^{u}_f + \lambda^{s}_f = \int_{\mathbb{T}^2} \log Jf dm = \log \sigma,$$ 
where $\sigma \equiv Jf = \vert \det A \vert$, then $\lambda_{f}^s(x) \equiv  \lambda_{A}^s$ for all $x \in \mathbb{T}^2$.

\subsection{Conservativeness}
\label{subsec:conserv}

Varão in \cite{varao2018rigidity} shows a similar result to Theorem \ref{teo:b} for conservative partial hyperbolic diffeomorphisms on $\mathbb{T}^3$. A generalization to our context, for hyperbolic endomorphisms on $\mathbb{T}^2$, requires some clarification on the role of conservativeness in the proof, so we recall this concept briefly.

A \textit{conservative} endomorphism on a Riemann manifold $M$ is a map that preserves the volume $m$. That is, $f \in C^1(M, M)$ is \textit{conservative} if $m(f^{-1}(A)) = m(A)$ for any measurable $A$. If $f$ is a diffeomorphism, then conservativeness is equivalent to $m(f(A)) = m(A)$, and by the change of variables formula, we have
\begin{equation}
    \label{eq:ch-var}
    m(f(A)) = \int_A \vert \det Df \vert \; dm,
\end{equation}
thus conservativeness is equivalent to the Jacobian $Jf = \vert \det Df \vert$ being constant equal to 1.

For endomorphisms, however, the formula (\ref{eq:ch-var}) only holds for small injectivity domains, and $m(f(A)) = m(A)$ is not equivalent to conservativeness, neither is the condition $Jf \equiv \mbox{constant}$. What we actually have for conservative endomorphisms is that
$$\sum_{x \in f^{-1}(\{y\})} \dfrac{1}{\vert \det Df(x) \vert} = 1.$$

In the case that $M$ is compact, conservativeness implies that there is $M > 1$ such that $1 \leq \vert \det Df(x) \vert \leq M$, and this condition is preserved when we lift the map to the universal cover.

Anosov endomorphisms, even if they are conservative, are volume expanding. Take for instance $f(x) = 2x \mod 1$ in $S^1$ or a linear toral endomorphism, that have constant Jacobian equal to the degree of the map, thus being conservative. Their lifts to the universal cover are not conservative.

\subsection{Theorem C}
\label{sec:teoc}

In our strategy to prove Theorem \ref{teo:b}, we need four different measures defined on the same local leaf to be equivalent, with Radon--Nikodym derivatives uniformly bounded ($m^k_x$, $\hat{\lambda}^k_x$, $F^k_*m^0_{F^{-k}(x)}$ and $F^k_*\hat{\lambda}^0_{F^{-k}(x)}$). Some of these equivalences is guaranteed by the UBD property, while others are a consequence of $Jf \equiv \mbox{constant}$. To obtain some of these equivalences, we introduce a weaker hypothesis than requiring $Jf$ to be constant. We define this condition as follows, and its role is better addressed during the proof.

\begin{definition}
    \label{def:hyp-c}
    For an $f$-invariant foliation, we say that $f$ \textit{quasi preserves densities along the foliation} if there is a constant $C >1$ such that, for each local leaf $\mathcal{W}$, and each $k \in \mathbb{N}$, 

    \begin{equation*}
        \label{eq:hyp-C}
        C^{-1} \leq \dfrac{d \hat{\lambda}^k_{f^{k}(x)}}{d f^k_*\hat{\lambda}^0_x} \leq C,
    \end{equation*}
    where $\hat{\lambda}^k_x$ is the normalized volume on $f^k(\mathcal{W})$.
\end{definition}

The constant $C$ in the definition does not depend on the leaf neither on $k$, that is, $\hat{\lambda}^k_{f^{k}(x)} \stackrel{u}{\sim}f^k_*\hat{\lambda}^0_x$ with uniform constant $C$ for the family of local measures $\hat{\lambda}^k_{f^{k}(x)}$, under the notation established after Definition \ref{def:UBD}.

Essentially, this hypothesis says that, by iterating with $f$, the densities of the induced volume are not distorted too much. That means, in the hyperbolic case, that the expansion/contraction seen on the leaf is ``well distributed'' along the leaf.

Quasi preservation of densities is more general than constant Jacobian, thus it is more general than conservativeness in the invertible setting. In the non invertible case, conservativeness is more general than constant Jacobian, and we do not know if there is a relation between conservativeness and quasi preservation of densities.

\begin{proof}[Proof of Theorem \ref{teo:c}]
We use the hypotheses on the unstable and stable foliations to prove that $\lambda_f^{u} \equiv \lambda_A^{u}$, and with an analogous argument we see that $\lambda_f^{s} \equiv \lambda_A^{s}$ as in Theorem \ref{teo:b}. The only difference from the proof of Theorem \ref{teo:b} is during Lemma \ref{lem:lemma2}, in which, instead of obtaining on Lemma \ref{lem:claim1} that

$$\dfrac{dm^k_x}{d F^k_*m^0_{F^{-k}(x)}} = 1,$$
we simply use the fact that quasi preservation of densities gives us
$$ C^{-1} \leq \dfrac{\hat{\lambda}^k_x}{d F^k_*\hat{\lambda}^0_{F^{-k}(x)}} \leq C,$$
which imply
\begin{equation}
    \label{eq:claim-qpd}
    C^{-3} \leq \dfrac{d F^k_*m^0_{F^{-k}(x)}}{d m^k_x} = \dfrac{d F^k_*m^0_{F^{-k}(x)}}{d F^k_*\hat{\lambda}^0_{F^{-k}(x)}} \dfrac{d F^k_*\hat{\lambda}^0_{F^{-k}(x)}}{d \hat{\lambda}^k_x} \dfrac{d \hat{\lambda}^k_x}{d m^k_x} \leq C^3,
\end{equation}
or $F^k_*m^0_{F^{-k}(x)} \stackrel{u}{\sim} m^k_x$ with uniform constant $C^3$. Indeed, the UBD property gives us that $m^k_x \stackrel{u}{\sim} \hat{\lambda}^k_x$ and $F^k_*m^0_{F^{-k}(x)} \stackrel{u}{\sim} F^k_*\hat{\lambda}^0_{F^{-k}(x)}$ both with uniform constant $C > 1$, and quasi preservation of densities imply $F^k_*\hat{\lambda}^0_{F^{-k}(x)} \stackrel{u}{\sim} \hat{\lambda}^k_x$ with the same uniform constant $C$, without loss of generality. Thus equation (\ref{eq:claim-qpd}) holds.

Then, the inequality (\ref{eq:bound-eta}) in this case is
$$C^{-4}K^{-1}d \lambda_x \leq d \eta^k_x \leq C^4K d \lambda_x,$$
and the rest of the proof is the same.

\end{proof}

\section*{Data availability}
No new data were created or analyzed in this study

\section*{Acknowledgements}
The authors would like to thank Boris Hasselblatt and Koichi Hiraide for kindly clarifying some points for Section \ref{sec:preli}. We also thank the anonymous referees for the careful reading and valuable comments. M. C. was partially financed by the Coordenação de Aperfeiçoamento de Pessoal de Nível Superior - Brasil (CAPES) - grant 88882.333632/2019-01 and the Fundação Carlos Chagas Filho de Amparo à Pesquisa of the State of Rio de Janeiro (FAPERJ) E-26/202.014/2022. R. V. was partially financed by CNPq and Fapesp grants 18/13481-0 and 17/06463-3.

\bibliographystyle{plain}
\bibliography{references}

\end{document}

%% file: lemma3.pdf_tex
\begingroup%
  \makeatletter%
  \providecommand\color[2][]{%
    \errmessage{(Inkscape) Color is used for the text in Inkscape, but the package 'color.sty' is not loaded}%
    \renewcommand\color[2][]{}%
  }%
  \providecommand\transparent[1]{%
    \errmessage{(Inkscape) Transparency is used (non-zero) for the text in Inkscape, but the package 'transparent.sty' is not loaded}%
    \renewcommand\transparent[1]{}%
  }%
  \providecommand\rotatebox[2]{#2}%
  \newcommand*\fsize{\dimexpr\f@size pt\relax}%
  \newcommand*\lineheight[1]{\fontsize{\fsize}{#1\fsize}\selectfont}%
  \ifx\svgwidth\undefined%
    \setlength{\unitlength}{566.92913386bp}%
    \ifx\svgscale\undefined%
      \relax%
    \else%
      \setlength{\unitlength}{\unitlength * \real{\svgscale}}%
    \fi%
  \else%
    \setlength{\unitlength}{\svgwidth}%
  \fi%
  \global\let\svgwidth\undefined%
  \global\let\svgscale\undefined%
  \makeatother%
  \begin{picture}(1,0.4)%
    \lineheight{1}%
    \setlength\tabcolsep{0pt}%
    \put(0,0){\includegraphics[width=\unitlength,page=1]{lemma3.pdf}}%
    \put(0.29382961,0.34547325){\color[rgb]{0,0,0}\makebox(0,0)[lt]{\lineheight{1.25}\smash{\begin{tabular}[t]{l}$\mathcal{B}_{\overline{\xi}}$\end{tabular}}}}%
    \put(0.1237747,0.16880371){\color[rgb]{0,0,0}\makebox(0,0)[lt]{\lineheight{1.25}\smash{\begin{tabular}[t]{l}$\overline{x}$\end{tabular}}}}%
    \put(0.1233,0.26310211){\color[rgb]{0,0,0}\makebox(0,0)[lt]{\lineheight{1.25}\smash{\begin{tabular}[t]{l}$\overline{y}$\end{tabular}}}}%
    \put(0,0){\includegraphics[width=\unitlength,page=2]{lemma3.pdf}}%
    \put(0.90132683,0.27063347){\color[rgb]{0,0,0}\makebox(0,0)[lt]{\lineheight{1.25}\smash{\begin{tabular}[t]{l}$F^n(\mathcal{B}_{\overline{\xi}})$\end{tabular}}}}%
    \put(0,0){\includegraphics[width=\unitlength,page=3]{lemma3.pdf}}%
    \put(0.54657318,0.25309166){\color[rgb]{0,0,0}\makebox(0,0)[lt]{\lineheight{1.25}\smash{\begin{tabular}[t]{l}$F^n(\overline{y})$\end{tabular}}}}%
    \put(0.61930225,0.14882202){\color[rgb]{0,0,0}\makebox(0,0)[lt]{\lineheight{1.25}\smash{\begin{tabular}[t]{l}$F^n(\overline{x})$\end{tabular}}}}%
    \put(0.86977936,0.23122139){\color[rgb]{0,0,0}\makebox(0,0)[lt]{\lineheight{1.25}\smash{\begin{tabular}[t]{l}$y_n$\end{tabular}}}}%
    \put(0.88434685,0.15670135){\color[rgb]{0,0,0}\makebox(0,0)[lt]{\lineheight{1.25}\smash{\begin{tabular}[t]{l}$w$\end{tabular}}}}%
    \put(0,0){\includegraphics[width=\unitlength,page=4]{lemma3.pdf}}%
    \put(0.41613964,0.22614052){\color[rgb]{0,0,0}\makebox(0,0)[lt]{\lineheight{1.25}\smash{\begin{tabular}[t]{l}$F^n$\end{tabular}}}}%
  \end{picture}%
\endgroup%